\documentclass[11pt,a4paper,leqno]{article}
\usepackage[latin1]{inputenc}
\usepackage[english]{babel}
\usepackage{amsmath}
\usepackage{amsfonts}
\usepackage{amssymb}
\usepackage{mathrsfs}
\usepackage{hyperref,comment} 
\usepackage[margin=1in]{geometry}
\usepackage{pifont}
\usepackage{pgf}
\usepackage{pgfplots}
\usepackage{tikz-3dplot}
\usepackage{graphicx}
\usepackage{fullpage}
\usepackage{color}
\usepackage{calrsfs}
\DeclareMathAlphabet{\pazocal}{OMS}{zplm}{m}{n}

\author{J\'er\^ome Bolte\footnote{TSE (GREMAQ, Universit\'{e} Toulouse I), Manufacture des Tabacs, 21 all\'{e}e de Brienne, 31015 Toulouse, France. E-mail: {\tt jerome.bolte@tse-fr.eu.} Effort sponsored by the Air Force Office of Scientific Research, Air Force Material Command, USAF, under grant number FA9550-14-1-0056. This research  also benefited from the support of the ``FMJH Program Gaspard Monge in optimization and operations research".} \, and Edouard Pauwels\footnote{Faculty of Industrial Engineering and Management, Technion, Haifa, Israel. E-mail: {\tt epauwels@ie.technion.ac.il}}}
\date{\today}
\title{Majorization-minimization procedures and convergence of SQP methods 
for semi-algebraic and tame programs }

\newcommand{\argmin}{{\rm argmin}}
\newcommand{\PP}{\big(\mathcal{P}\big)}
\newcommand{\PPx}{\mathcal{P}(x)}
\newcommand{\NLP}{\big(\mathcal{P}_{\scriptstyle\rm \,NLP}\big)}
\newcommand{\SSS}{\big(\mathcal{S}\big)}
\newcommand{\DD}{\mathcal{D}}
\newcommand{\FF}{\mathcal{F}}
\newcommand{\dist}{\mbox{\rm dist}\,}

\newcommand{\dom}{\mbox{dom}\,}
\newcommand{\RR}{\mathbb{R}}
\newcommand{\R}{\mathbb{R}}

\newcommand{\NN}{\mathbb{N}}

\newcommand{\test}{{\rm test}}
\newcommand{\cone}{\mbox{\rm cone}\,}
\newcommand{\co}{\mbox{\rm co}\,}
\newcommand{\val}{\mbox{\rm val}\,}

\newtheorem{theorem}{Theorem}[section]
\newtheorem{lemma}[theorem]{Lemma}
\newtheorem{proposition}[theorem]{Proposition}
\newtheorem{corollary}[theorem]{Corollary}

\newtheorem{definition}{Definition}

\newtheorem{remark}{Remark}

\newtheorem{example}{Example}

\newenvironment{proof}[1][]{\noindent {\bf Proof. #1\;}}{\hfill $\Box$\\}

\colorlet{FRAME}{yellow!5!white}

\begin{document}

\maketitle

\begin{abstract}
	In view of solving nonsmooth and nonconvex problems involving complex constraints (like standard NLP problems), we study general maximization-minimization procedures produced by families of strongly convex sub-problems.  Using techniques from semi-algebraic geometry and variational analysis --in particular \L ojasiewicz inequality-- we establish the convergence of sequences generated by this type of schemes to critical points. 
	
	The broad applicability of this process is illustrated in the context of NLP. In that case critical points coincide with KKT points. When the data are semi-algebraic or real analytic our method applies (for instance) to the study of various SQP methods: the moving balls method, S$\ell^1$QP, ESQP. Under standard qualification conditions, this provides --to the best of our knowledge-- the first general convergence results  for  general nonlinear programming problems.  We emphasize the fact that, unlike most works on this subject,   no second-order conditions  and/or convexity assumptions whatsoever are made. Rate of convergence are shown to be of the same form as those commonly encountered with first order methods.
	\end{abstract}

\newpage

\tableofcontents

\section{Introduction}

Many optimization methods consist in approximating a given problem by a sequence of simpler problems that can be solved in closed form or computed fast, and which eventually provide a solution or some acceptable improvement. From a mathematical viewpoint some central questions are the convergence of the sequence to a desirable point (e.g. a KKT point), complexity estimates, rates of convergence. For these theoretical purposes  it is often assumed that the constraints are simple, in the sense that their projection  is easy to compute (i.e. known through a closed formula), or that the objective involve nonsmooth terms whose proximal operators are available (see e.g. Combettes and Pesquet (2011) \cite{combpesq11}, Attouch et al. (2010) \cite{attouch2010proximal}, Attouch et al. (2013) \cite{attouch2013convergence}).  An important challenge is to go beyond this {\em prox friendly} (\footnote{A term we borrow from Cox et al. recent developments (2013) \cite{nemjud}}.) setting and to address mathematically the issue of nonconvex nonsmooth problems presenting complex geometries. 

The richest field in which these problems are met, and which was the principal motivation to  this research, is probably ``standard nonlinear programming" in which KKT points are generally sought through the resolution of quadratic programs of various sorts.  We shall refer here to these methods under the general vocable of {\em SQP methods}. The bibliography on the subject is vast, we refer the readers to Bertsekas (1995) \cite{bert95}, Nocedal and Wright (2006) \cite{nocewrig06}, Gill and Wong (2012) \cite{philwong12}, Fletcher et al. (2002) \cite{fletcheretcie} (for instance) and references therein for an insight. Although these methods are quite old now --the pioneering work seems to originate in the PhD thesis of Wilson \cite{wils63} in 1963-- and massively used in practice, very few general convergence or complexity results are available. Most of them are local and are instances of the classical case of convergence of Newton's method~(\footnote{Under second-order conditions and assuming that no Maratos effect \cite{mara78} troubles the process.}) Fletcher et al. (2002) \cite{Fletcher}, Bonnans et al. (2003) \cite{bonnans03}, Nocedal and Wright (2006) \cite{nocewrig06}. Surprisingly the mere question ``are the limit points KKT points?" necessitates rather strong assumptions and/or long developments -- see Bonnans et al. (2003) \cite[Theorem 17.2]{bonnans03} or Burke and Han (1989) \cite{burke}, Byrd et al. (2005) \cite{nocedal05}, Solodov (2009) \cite{solo09} for the drawbacks of ``raw SQP" in this respect, see  also Bertsekas (1995) \cite{bert95}, Bonnans et al. (2003) \cite{bonnans03} for some of the standard conditions/\-cor\-rec\-tions/re\-cipes ensuring that limit points are KKT.

Works in which actual convergence  (or even limit point convergence) are obtained under minimal assumptions seem to be pretty scarce. In \cite{fukuluotsen03} (2003), Fukushima et al. provided a general SQCQP  method (\footnote{Sequential quadratically constrained quadratic programming.}) together with a convergence result in terms of limit points, the results were further improved and simplified by Solodov (2004) \cite{sol04}. More related to the present work is the contribution of Solodov (2009) \cite{solo09}, in which a genuinely non-trivial proof for a SQP method to eventually provide KKT limit points is given. More recently, Auslender (2013) \cite{auslender2013extended} addressed the issue of the actual convergence in the convex case by modifying and somehow reversing the basic SQP protocol: ``the merit function" (see Han (1977) \cite{han77}, Powell (1973) \cite{powe73}) is directly used to devise descent directions as in Fletcher's pioneering S$\ell^1$QP method (1985) \cite{fletcher85}. In this line of research one can also quote the works of Auslender et al. (2010) \cite{auslender2010moving} on the ``moving balls" method -- another instance of the class of SQCQP methods.

Apart from Auslender (2013) \cite{auslender2013extended}, Auslender et al. (2010) \cite{auslender2010moving}, we are not aware of other results  providing actual convergence for general smooth convex functions (\footnote{We focus here on SQP methods but alternative methods for treating complex constraints are available, see e.g. Cox et al. (2013) \cite{nemjud} and references therein.}). After our own unfruitful tries,  we think this is essentially  due to the fact that the dynamics of active/inactive constraints is not well understood -- despite some recent breakthroughs Lewis (2002) \cite{lew02}, Wright (2003) \cite{wrig03}, Hare and Lewis (2004) \cite{harelewi04} to quote a few.  In any cases ``usual" methods for convergence or complexity fail  and to our knowledge there are very few other works on the topic. In the nonconvex world, the recent advances  of Cartis et al. (2014) \cite{toint14} are first steps towards a complexity theory for NLP. Since we focus here on convergence our approach is pretty different but obviously connections and complementarities must be investigated.

\medskip

Let us describe our method for addressing these convergence issues. Our approach is  threefold:
\begin{itemize}
\item[--] We consider nonconvex, possibly nonsmooth, semi-algebraic/real analytic data; we actually provide results for definable sets. These model many, if not most, applications.
\item[--] Secondly, we delineate and study a wide class of {\em majorization-minimization  methods} for nonconvex nonsmooth constrained problems. Our main assumption being that the procedures involve {\em locally Lipschitz continuous, strongly convex upper approximations}.\\
	Under  a  general qualification assumption, we establish the convergence of the process. Once more, nonsmooth Kurdyka-{\L}ojasiewicz (KL) inequality (\L{}ojasiewicz (1963) \cite{loja1963propriete}, Kurdyka (1998) \cite{kurdyka1998gradients}) appears as an essential tool.

\item[--] Previous results are applied to derive convergence of  SQP methods (Fletcher's S$\ell^1$QP (1985) \cite{fletcher85}, Auslender (2013) \cite{auslender2013extended}) and SQCQP methods (moving balls method Auslender et al. (2010) \cite{auslender2010moving}). To the best of our knowledge, these are the first general nonconvex results dealing with possibly large problems with complex geometries  -- which are not ``prox-friendly". Convergence rates have the form $O\left(\frac{1}{k^\gamma}\right)$ with $\gamma>0$.

\end{itemize}

We describe now these results into more details which will also give an insight at the main results obtained in this paper.

\medskip

\noindent
{\bf Majorization-minimization procedures (MMP)}. These methods consist in devising at each point of the objective a simple upper  model (e.g. quadratic forms) and to minimize/update these models dynamically in order to produce minimizing/descent sequences. This principle can be traced back, at least, to Ortega (1970) \cite[section 8.3.(d)]{ortega1970iterative} and have found many applications since then, mostly in the statistics literature Dempster et al. (1977) \cite{dempster1977maximum}, but also in other branches like recently in imaging sciences Chouzenoux et al. (2013) \cite{chouzenoux}. In the context of optimization, many iterative methods follow this principle, see Beck and Teboulle (2010) \cite{beck2010gradient}, Mairal (2013) \cite{mairal2013optimization} for numerous examples --and also Noll (2014) \cite{noll} where KL inequality is used to solve nonsmooth problems using a specific class of models. These procedures, which we have studied as tools,  appeared to have an interest for their own sake. Our main results in this respect are self-contained and can be found in Sections~\ref{sec:mainResult} and \ref{sec:convergence}. Let us briefly sketch a description of the MM models we use.\\
Being given a   problem of the form
\begin{equation*}
\PP	\qquad \min \Big\{f(x):x\in \DD\Big\}
\end{equation*}
where $f:\R^n\to \R$ is a  semi-algebraic continuous function and $\mathcal{D}$ is a nonempty closed semi-algebraic set, we define at each feasible point $x$, local semi-algebraic convex models for $f$ and $\DD$, respectively $h(x,\cdot):\R^n\to \R$ -- which is actually strongly convex-- and $D(x)\subset \R^n$. We then iteratively solve problems of the form
$$x_{k+1}=p(x_k):=\argmin\Big\{ h(x_k, y):y\in D(x_k)\Big\}, \: k\in \mathbb{N}.$$
An essential assumption is that of using {\em upper approximations} (\footnote{Hence the wording  of majorization-minimization method}):  $D(x)\subset \DD$ and $h(x,\cdot)\geq f(\cdot)$ on $D(x)$. When assuming semi-algebraicity of the various ingredients, convergence cannot definitely be seen as a consequence of the results in Attouch et al. (2013) \cite{attouch2013convergence}, Bolte et al. (2013) \cite{bolte2013proximal}. This comes from several reasons. First, we do not have a ``proper (sub)gradient method" for $\PP$ as required in the general protocol described in Attouch et al. (2013) \cite{attouch2013convergence}. A flavor of these difficulty is easily felt when considering  SQP. For these methods there is, at least apparently, a sort of an unpredictability of future active/inactive constraints:  the descent direction does not allow to forecast future activity and thus does not necessarily mimic an adequate   subgradient of $f+i_D$ or of similar aggregate costs~(\footnote{$i_D$ denotes the indicator function as defined in Section \ref{s:nonsmooth}}). Besides,  even when a better candidate for being the descent function is identified, explicit features inherent to the method  still remain to be  dealt with. 

The cornerstone of our analysis is the introduction and the study of the {\em value (improvement) function} $F(x)=h(x,p(x))$. It helps circumventing the possible anarchic behavior of active/inactive constraints by an implicit inclusion of  future progress within the cost. We establish that the sequence $x_k$ has a behavior very close to a subgradient method for $F$, see Section~\ref{sec:abstrConv}.

Our main result is an asymptotic alternative, a phenomena already guessed in Attouch et al. (2010) \cite{attouch2010proximal}: either the sequence $x_k$ tends to infinity, or it converges to a critical point. As a consequence, we have convergence of the sequence to {\em a single point} whenever the problem $\PP$ is coercive.

\medskip

\noindent
{\bf Convergence of SQP type methods.} The previous results can be applied to many algorithms (see e.g. Attouch et al. (2010) \cite{attouch2010proximal}, Bolte et al. (2013) \cite{bolte2013proximal}, Chouzenoux et al. (2013) \cite{chouzenoux}), but we concentrate on some SQP methods for which such results are novel. In order to avoid a too important concentration of hardships, we do not discuss here computational issues of the sub-steps, the prominent role of step sizes, the difficult question of the feasibility of sub-problems, we refer the reader to Fletcher (2000) \cite{Fletcher}, Bertsekas (1995) \cite{bert95}, Gill and Wong (2012) \cite{philwong12} and references therein. We would like also to emphasize that, by construction, the methods we investigate may or may not involve hessians in their second-order term but they must  systematically include a fraction of the identity as a regularization parameter, \`a la Levenberg-Marquardt (see e.g. Nocedal and Wright (2006) \cite{nocewrig06}). Replacing Hessian terms by their regularization or by fractions of the identity is a common approach to regularize ill-posed problems; it is also absolutely crucial when facing large scale problems see e.g. Gill et al. (2005) \cite{snopt}, Svanberg (2002) \cite{sva02}.

The aspects we just evoked above have motivated our choice of Auslender SQP method  and of the moving balls method which are both relatively ``simple" SQP/SQCQP methods. To show the flexibility of our approach, we also study a slight variant of S$\ell^1$QP, Fletcher (1985) \cite{fletcher85}. This method, also known as ``elastic SQP", is a modification of SQP making the sub-problems feasible by the addition of slack variables. In Gill et al. (2005) \cite{snopt} the method has been adapted and redesigned to solve large scale problems (SNOPT); a solver based on this methodology is available. 

For these methods, we show that a bounded sequence must converge to a single KKT point, our results rely only on semi-algebraic techniques and do not use  convexity nor second order conditions. The semi-algebraic assumption can be relaxed to definability or local definability (tameness, see Ioffe (2009) \cite{ioffe2009invitation} for an overview). We also establish that these methods come with convergence rates  similar to those observed in classical first-order method (Attouch and Bolte (2009) \cite{attouch2009convergence}). Finally, we would like to stress that the analysis relies on geometrical tools which have a long history in the study of convergence of dynamical systems of gradient type, see for e.g. \L{}ojasiewicz (1963) \cite{loja1963propriete}, Kurdyka (1998) \cite{kurdyka1998gradients}.

\medskip

\noindent
{\bf Organization of the paper.} Section \ref{sec:section2} presents our main results concerning SQP methods. In Section \ref{sec:mainResult}, we describe an abstract framework for majorization-minimization methods that is used to analyze the algorithms presented in Section \ref{sec:section2}. We give in particular a  general result on the convergence of MM methods. Definitions, proofs and technical aspects  can be found in Section \ref{sec:convergence}.  Our results on MM procedures and SQP are actually valid for the broader class of real analytic or definable data, this is explained in Section \ref{sec:representable}. The Appendix (Section \ref{sec:appendix}) is devoted to the technical study of SQP methods, it is shown in particular how they can be interpreted as MM processes.

\section{Sequential quadratic programming for semi-algebraic and tame problems}\label{s:nlp}
\label{sec:section2}

We consider in this section problems of the form:
\begin{equation}
	\label{eq:nlp}
	\begin{array}{llll}
		\NLP &  & \min_{x \in \RR^n} & f(x) \\
		&& \mathrm{s.t.} & f_i(x) \leq 0, \; i=1, \ldots, m \\
		&&& x \in Q
	\end{array}
\end{equation}
where each $f_i$ is twice continuously differentiable and $Q$ is a  nonempty closed convex set. $Q$ should be thought as a ``simple" set, i.e.,  a set whose projection is known in closed form (or ``easily" computed), like for instance one of the archetypal self dual cone $\R_+^n$, second order cone, positive semi-definite symmetric cone~(\footnote{Computing the projection in that case requires to compute eigenvalues, which may be very hard for large size problems}), but also an affine space, an $\ell^1$ ball, the unit simplex, or a box.  Contrary to $Q$, the set
$$\mathcal{F}  = \{x,\; f_i(x) \leq 0,\; i = 1, \ldots, m\},$$
has, in general, a complex geometry and its treatment necessitates local approximations in the spirit of SQP methods. 
Specific assumptions regarding coercivity, regularity and constraint qualification are usually required in order to ensure correct behavior of numerical schemes, we shall make them precise for each method we present here.  Let us simply recall that under these assumptions, any  minimizer $x$ of $\NLP$ must satisfy the famous KKT conditions:
\begin{align}
 & x\in Q, \; f_1(x)\leq 0,\ldots,f_m(x)\leq 0,\\
 & \exists \: \lambda_1\geq  0, \ldots, \lambda_m\geq 0,\\
 & \nabla f(x)+\sum \lambda_i \nabla f_i(x)+N_Q(x)\ni 0,\\
& \lambda_if_i(x)=0, \forall i=1,\ldots, m,
\end{align}
 where $N_Q(x)$ is the normal cone to $Q$ at $x$ (see Section~\ref{sec:tools}).

SQP methods assume very different forms, we pertain here to three ``simple models" with the intention of illustrating the versatility of our approach (but other studies could be led):

\begin{itemize}
 \item[--] Moving balls method: an SQCQP method,
\item[--] ESQP method: a merit function approach with $\ell^\infty$ penalty,
\item[--] S$\ell^1$QP method: a merit function approach with $\ell^1$ penalty.
\end{itemize}

\subsection{A sequentially constrained quadratic method: the moving balls method}
This method was introduced in Auslender et al. (2010) \cite{auslender2010moving} for solving problems of the form of (\ref{eq:nlp}) with $Q = \RR^n$. The method enters the framework of sequentially constrained quadratic problems. It  consists in approximating the original problem by a sequence of quadratic problems over an intersection of balls. Strategies for simplifying constraints approximation, computations of the intermediate problems are described in Auslender et al. (2010) \cite{auslender2010moving}, we only focus here on the convergence properties and rate estimates. The following assumptions are necessary.

\begin{description}
	\item[Regularity:] The functions 
		\begin{equation} \label{eq:MB1}
			f, f_1, \ldots, f_m:\R^n\rightarrow \R
		\end{equation}		are $C^2$, with Lipschitz continuous gradients. For each $i=1,\ldots,m$, we denote by $L_i>0$ some Lipschitz constants of $\nabla f_i$ and by $L>0$  a Lipschitz constant of $\nabla f$. 		
		\item[Mangasarian-Fromovitz Qualification Condition (MFQC):]  For $x$ in $\mathcal{F}$,  set  $I(x)=\{i=1,\ldots,m: f_i(x)=0\}.$  MFQC writes
		\begin{equation} \label{eq:MB3}
		\forall x\in \FF, \: \exists d \in \RR^n \text{ such that } \left\langle\nabla f_i(x), d \right\rangle < 0, \forall i\in I(x).
		\end{equation}
	\item[Compactness:] There exists a feasible $x_0 \in \mathcal{F}$ such that 
		\begin{equation} \label{eq:MB2}
			\{x\in \R^n: f(x) \leq f(x_0)\} \mbox{ 	is bounded.}
		\end{equation}
	
\end{description}
\begin{remark}
	As presented in Auslender et al. (2010) \cite{auslender2010moving}, the moving ball method is applicable to functions that are only $C^1$ with Lipschitz continuous gradient. The assumption made in (\ref{eq:MB1}) is therefore slightly more restrictive than the original presentation of Auslender et al. (2010) \cite{auslender2010moving}. 
\end{remark}

The moving balls  method is obtained by solving a sequence of quadratically constrained problems.\\
\centerline{
\fcolorbox{black}{FRAME}{
\begin{minipage}{13cm}
\begin{center}{\bf Moving balls method }
\end{center}
\begin{equation*}
	\begin{array}{ll}
		\text{\bf Step 1 } & x_0\in \FF.\\
		\text{\bf Step 2 } & \text{Compute }x_{k+1} \text{ solution of }\\
		 &   
		 \begin{array}{rl}
			 \displaystyle\min_{y\in \RR^n}&f(x_{k}) + \left\langle \nabla f(x_k), y - x_k \right\rangle + \frac{L}{2} ||y - x_k||^2\\
			 \text{s.t.}&f_i(x_k) + \left\langle \nabla f_i(x_k), y - x_k \right\rangle + \frac{L_i}{2} ||y - x_k||^2  \leq 0,\,i = 1 \ldots m
			\end{array}
	\end{array}
\end{equation*}
\end{minipage}
}
}
\medskip

The algorithm can be proven to be well defined and to produce  a feasible method provided that $x_0$ is feasible, i.e., 
$$x_k\in \FF, \forall k\geq 0.$$
These aspects are thoroughly   described in Auslender et al. (2010) \cite{auslender2010moving}. 

\begin{theorem}[Convergence of the moving balls method]
Recall that $Q=\RR^n$ and assume that the following conditions hold
	\begin{itemize}
	\item[--] The functions $f,f_1,\ldots,f_m$ are semi-algebraic,
	\item[--] Lipschitz continuity conditions \eqref{eq:MB1}, 
	\item[--] Mangasarian-Fromovitz qualification condition \eqref{eq:MB3},
		\item[--] boundedness condition  \eqref{eq:MB2},
		\item[--] feasibility of the starting point $x_0 \in \mathcal{F}$.	
		\end{itemize} 
		Then, 		
\begin{itemize}	
	\item[(i)] The sequence $\{x_k\}_{k \in \NN}$  defined by the moving balls method converges to a feasible point $x_\infty$ satisfying the KKT conditions for the nonlinear programming problem $\NLP$.
	\item[(ii)]   Either convergence occurs in a finite number of steps or the rate is of the form:
 
(a)  $\|x_k-x_\infty\|=O(q^k)$, with $q\in(0,1)$, 
 
(b)  $\|x_k-x_\infty\|=O\left(\frac{1}{k^{\gamma}}\right)$, with $\gamma>0$. 
	
\end{itemize}
\end{theorem}

\subsection{Extended sequential quadratic method}


ESQM method (and S$\ell^1$QP) grounds on the well known observation that an NLP problem can be reformulated as an ``unconstrained problem" involving an exact penalization. 
Set $f_0={\mathbf 0}$ and consider 
\begin{align}
	\label{eq:penalization}
	\min_{x \in Q} \: \left\{f(x) + \beta \max_{i = 0, \ldots,m} f_i(x)\right\}
\end{align}
where $\beta$ is positive parameter. Under mild qualification assumptions and for $\beta$ sufficiently large, critical points of the above are KKT points of the initial nonlinear programming $\NLP$.  
Building on this fact, ESQM is devised as follows:
\begin{itemize}
	\item At a fixed point $x$ (non necessarily feasible), form a model of (\ref{eq:penalization}) such that: 
	\begin{itemize}
		\item complex terms $f,f_1,\ldots,f_m$ are linearized,
		\item a quadratic term $\frac{\beta}{2} ||y - x||^2$ is added both for conditioning and local control,
	\end{itemize}
\item minimize the model to find a descent direction and perform a step of size $\lambda > 0$,
\item both terms $\lambda,\beta$ are adjusted online:  
\begin{itemize}
\item $\lambda$ is progressively made smaller to ensure a descent condition,
\item $\beta$ is  increased to eventually reach a threshold for exact penalization.
\end{itemize}
\end{itemize}
We propose here a variation of this method which consists in modifying the quadratic penalty term instead of relying on a line search procedure to ensure some sufficient decrease. For a fixed $x$ in $\R^n$, we consider a local model of the form:
$$
\begin{aligned}
 h_{\beta}(x, y) &= f(x) + \left\langle\nabla f(x), y - x  \right\rangle+ \beta \max_{i=0, \ldots, p}\left\{f_i(x) + \left\langle\nabla f_i(x), y - x \right\rangle \right\} \\
& \phantom{=} + \frac{(\lambda +\beta\lambda')}{2} ||y - x||^2 + i_Q(y),
\end{aligned}
$$
where $\beta$ is a parameter and $\lambda,\lambda'>0$ are fixed. 

As we shall see this model is to be iteratively used to provide descent directions and ultimately KKT points. Before describing into depth the algorithm,  let us state our main assumptions (recall that $\FF=\{x\in \R^n:f_i(x)\leq0, \forall i=1,\ldots,m \}$).
\begin{description}
\item[Regularity:] The functions
	\begin{equation}\label{eq:ESQM0}
		f, f_1, \ldots, f_m:\R^n\rightarrow \R
	\end{equation}		are $C^2$, with Lipschitz continuous gradients. For each $i=1,\ldots,m$, we denote by $L_i>0$ some Lipschitz constants of $\nabla f_i$ and by $L>0$  a Lipschitz constant of $\nabla f$. We also assume that the step size parameters  satisfy
\begin{equation}\label{aus-lip}\lambda \geq L\text{ and }\lambda'\geq\max_{i=1, \ldots, m} L_i.\end{equation}
	\item[Compactness:] For all real numbers $\mu_1,\ldots,\mu_m$, the set 
		\begin{equation}\label{eq:ESQM3}
			\{x \in Q,\;f_i(x) \leq \mu_i,\; i = 1, \ldots, m\}\text{ is compact.} 
		\end{equation}
	\item[Boundedness:]  
		\begin{equation}\label{eq:ESQM4}
			\inf_{x\in Q} f(x) > -\infty.
		\end{equation}
	\item[Qualification condition:] The function 
		$$\max_{i=1,\ldots,m} f_i+i_Q$$ 
		has {\em no critical points} on the set $\Big\{x\in Q: \exists i=1,\ldots,m, f_i(x)\geq0\Big\}$. \\
		Equivalently, $\forall x \in \{x\in Q: \exists i=1,\ldots,m, f_i(x)\geq0\}$, there cannot exist $\{u_i\}_{i \in I}$ such that
		\begin{equation}\label{eq:ESQM2}
			u_i \geq 0,\; \sum_{i \in I} u_i = 1,\; \; \sum_{i \in I}\left\langle u_i \nabla f_i(x), z - x \right\rangle \geq 0,\, \forall z \in Q,
		\end{equation}
		where $I = \Big\{j>0,\; f_j(x) = \max_{i=1, \ldots, m} \{f_i(x)\} \Big\}.$
	\begin{remark}\label{r:qualif}
	{\rm (a) Set $J=\{1,\ldots,m\}$. The intuition behind this condition is simple: $\max_{i\in J} f+i_Q$ is assumed to be (locally) ``sharp" and thus $\beta\max_{i\in J} (\mathbf{0},f_i)+i_Q$ resembles $i_{Q\cap\FF}$ for big $\beta$. \\
	(b) The condition \eqref{eq:ESQM2} implies the generalized Mangasarian-Fromovitz condition (sometimes called Robinson condition):
	     \begin{equation*}
			\forall x \in Q \cap \mathcal{F}, \exists y \in Q\setminus\{x\},\;  \left\langle\nabla f_i(x), y - x  \right\rangle < 0, \forall i = 1, \ldots, m,\; \text{ such that }f_i(x)=0.
		\end{equation*}
		(c)  Assume $Q=\R^n$. The qualification condition \eqref{eq:ESQM2} implies that the feasible set is connected, which is a natural extension of the more usual convexity assumption.    [Proof. Argue by contradiction and assume that the feasible set has at least two connected components. Take two points $a,b$ in each of these components. The function  $g=\max\{ f_i: i=1,\ldots,m\}$ satisfies $g(a)=g(b)=0$. Using the compactness assumption (\ref{eq:ESQM3}), the conditions of the mountain pass theorem (Shuzhong (1985) \cite[Theorem 1]{shuzhong1985ekeland}) are thus satisfied. Hence, there exists a critical point $c$ such that $g(c)>0$ (strictly speaking, this is a Clarke critical point, but in this specific setting, this corresponds to the notion of crtitical point we use un this paper see Rockafellar and Wets (1998) \cite[Theorem 10.31]{rockafellar1998variational}). Thence $c$ is non feasible and the criticality of $c$ contradicts our qualification assumption.]}
	\end{remark}
\end{description}
\medskip

Let us finally introduce {\em feasibility test functions}
\begin{equation}\label{test}
\test_i(x,y)=f_i(x)+\langle\nabla f_i(x),y-x\rangle
\end{equation}
for all $i=1,\ldots,m$ and $x,y$ in $\R^n$. 

\begin{remark}[Online feasibility test]{\rm We shall use the above functions for measuring the quality of $\beta_k$. 
These tests function will also be applied to the analysis of S$\ell^1$QP. Depending on the information provided by the algorithm, other choices could be done, as for instance $\test_i(x,y)=f_i(x)+\langle\nabla f_i(x),y-x\rangle+\frac{L_{f_i}}{2}\|y-x\|$ or simply $\test_i(x,y)=f_i(y)$. }
\end{remark}

We proceed now to the description of  the algorithm.
\medskip

\centerline{\fcolorbox{black}{FRAME}{ 
\begin{minipage}{13cm}
\begin{center}{\bf Extended Sequential Quadratic Method (ESQM) }
\end{center}
\begin{equation*}
	\begin{array}{ll}
		\text{\bf Step 1 }& \text{Choose }x_0\in Q,\; \beta_0,\, \delta>0 \\
		\text{\bf Step 2 } & \text{Compute the unique solution $x_{k+1}$  of } \min_{y\in \RR^n} h_{\beta_k}(x_k, y),\\
		& \text{ i.e. solve for }y \text{ (and }s) \text{ in:}\\
		& 
		\begin{array}{rl}
			\displaystyle\min_{y, s}&f(x_k)+\langle \nabla f(x_k),y-x_k\rangle + \beta_ks +\frac{\left(\lambda+\beta_k\lambda'\right)}{2}||y-x||^2\\
			\text{s.t.}& \,f_i(x_k)+\langle \nabla f_i(x_k), y-x_k\rangle\leq s, i = 1 \ldots m,\\
			   & s \geq 0 \\
			   & y\in Q.
		\end{array}\\
		\text{\bf Step 3 } & \text{ If $\test_i(x_k,x_{k+1})\leq 0$ for all $i=1,\ldots,m$, then $\beta_{k+1}=\beta_k$,}\\
			&  \text{ otherwise }\beta_{k+1}=\beta_k+\delta  
	\end{array}
\end{equation*}
\end{minipage}
}
}
\medskip

\begin{remark}\label{auslender}
	{\rm (a) 
	Working with quadratic terms involving Hessians in $h_{\beta_k}$ is possible provided that local models are upper approximations (one can work for instance with approximate functions \`a la Levenberg-Marquardt Nocedal and Wright (2006) \cite{nocewrig06}).\\
(b)	The algorithm presented in Auslender (2013) \cite{auslender2013extended} is actually slightly different from the one above. Indeed, the quadratic penalty term was there simply proportional to $\beta$ and the step sizes were chosen by line search. Original ESQP could thus be  seen as a kind of backtracking version of the above method. \\
(c) Let us also mention that many updates rules are possible, in particular rules involving upper bounds of local Lagrange multipliers. The essential aspect is that exact penalization is reached in a finite number of iterations.\\
(d) Observe that the set $Q$ of simple constraints is kept as is in the sub-problems.
\\

}
\end{remark}

The convergence analysis carried out in Auslender (2013) \cite{auslender2013extended} can be extended to our setting, leading to the following theorem (note we do not use the semi-algebraicity assumptions).
\begin{theorem}[Auslender (2013) \cite{auslender2013extended}]
	\label{th:auslender}
	Assume that the following properties hold
	\begin{itemize}
			\item[--] Lipschitz continuity conditions \eqref{eq:ESQM0}, 
			\item[--] steplength conditions \eqref{aus-lip},
		\item[--]  qualification assumption \eqref{eq:ESQM2}, 
		\item[--] boundedness assumptions \eqref{eq:ESQM3}, \eqref{eq:ESQM4},
	\end{itemize}
	then the sequence of parameters $\beta_k$ stabilizes after a finite number of iterations $k_0$ and all cluster points of the sequence $\{x_k\}_{k \in \NN}$ are KKT points of the nonlinear programming problem $\NLP$.
\end{theorem}

\medskip

The application of the techniques developed in this paper allow to prove a much stronger result:

\begin{theorem}[Convergence of ESQM]
	Assume that the following conditions hold
	\begin{itemize}
	\item[--] The functions $f,f_1,\ldots,f_m$ and the set $Q$ are real semi-algebraic,
		\item[--] Lipschitz continuity condition  \eqref{eq:ESQM0}, 
	\item[--] steplength condition \eqref{aus-lip},
		\item[--]  qualification assumption  \eqref{eq:ESQM2}, 
		\item[--] boundedness assumptions \eqref{eq:ESQM3}, \eqref{eq:ESQM4},
		\end{itemize}
 Then, 
 \begin{itemize}
 \item[(i)] The sequence $\{x_k\}_{k \in \NN}$ generated by {\rm (ESQM)} converges to a feasible point $x_\infty$ satisfying the KKT conditions for the nonlinear programming problem $\NLP$. 
 \item[(ii)] Either convergence occurs in a finite number of steps or the rate is of the form:
 
(a)  $\|x_k-x_\infty\|=O(q^k)$, with $q\in(0,1)$, 
 
(b)  $\|x_k-x_\infty\|=O\left(\frac{1}{k^{\gamma}}\right)$, with $\gamma>0$. 
\end{itemize}	

\end{theorem}

This result gives a positive answer to the ``Open problem 3" in Auslender (2013) \cite[Section 6]{auslender2013extended} (with a slightly modified algorithm).

\subsection{S$\ell^1$QP aka ``elastic sequential quadratic method"}


The S$\ell^1$QP  is an $\ell^1$ version of the previous method. It seems to have been introduced in the eighties by Fletcher\cite{fletcher85}.  Several aspects of this method are discussed in Fletcher (2000) 	\cite{Fletcher}; see also Gill et al. (2005) \cite{snopt} for its use in the resolution of large size problems (SNOPT algorithm). 
The idea is based this time on the minimization of the $\ell^1$ penalty function:
\begin{align}
	\label{flet}
	\min_{x \in Q}\, f(x) + \beta \sum_{i = 0}^m f^+_i(x)
\end{align}
where $\beta$ is a positive parameter and where  we have set $a^+=\max(0,a)$ for any real number $a$. Local models are of the form:
$$
\begin{aligned}
&\phantom{=}&&  h_{\beta}(x, y) \\
&= && f(x) + \left\langle\nabla f(x), y - x  \right\rangle + \beta \sum_{i=1}^m\left(f_i(x) + \left\langle\nabla f_i(x), y - x \right\rangle \right)^+ \\
&\phantom{=}&& + \frac{(\lambda +\beta\lambda')}{2} ||y - x||^2 + i_Q(y), \quad \forall x,y\in \R^n,
\end{aligned}
$$
where as previously $\lambda,\lambda'>0$ are fixed parameters.  Using slack variables the minimization of $h_{\beta}(x,.)$ amounts to solve the problem 
\begin{equation*}
	\begin{array}{rl}
			\displaystyle\min& f(x)+\langle \nabla f(x),y-x\rangle + \beta \sum_{i=1}^ms_i+\frac{\left(\lambda+\beta\lambda'\right)}{2}||y-x||^2\\
			\text{s.t.}& f_i(x)+\langle \nabla f_i(x), y-x\rangle \leq s_i,\, i=1 \ldots m\\
 			& s_1,\ldots,s_m \geq 0 \\
 			& y\in Q.
		\end{array}\\
\end{equation*}
Once again, the above is very close to the ``usual" SQP step, the only difference being the elasticity conferred to the constraints by the penalty term.

 The main requirements needed for this method are quasi-identical to those we used for ESQP: we indeed assume \eqref{eq:ESQM0}, \eqref{eq:ESQM2}, \eqref{eq:ESQM3}, \eqref{eq:ESQM4}, while \eqref{aus-lip} is replaced by:
\begin{equation}\label{fletcher-lip}\lambda \geq L\text{ and }\lambda'\geq\sum_{i=1}^ mL_i.\end{equation}
The latter is  more restrictive in the sense that smaller step lengths are required, but on the other hand this restriction comes with more flexibility in the relaxation of the constraints.


\medskip

In the description of the algorithm below, we make use the test functions \eqref{test} described in the previous section.

\medskip
\centerline{\fcolorbox{black}{FRAME}{ 
\begin{minipage}{13cm}
\begin{center}{\bf S$\ell^1$QP}
\end{center}
\begin{equation*}
	\begin{array}{ll}
		\text{\bf Step 1 }& \text{Choose }x_0\in Q,\; \beta_0,\, \delta>0 \\
		\text{\bf Step 2 } & \text{Compute the unique solution $x_{k+1}$  of } \min_{y\in \RR^n} h_{\beta_k}(x_k, y),\\
		& \text{ i.e. solve for }y \text{ (and }s) \text{ in:}\\
		& 
		\begin{array}{rl}
			\displaystyle\min& f(x)+\langle \nabla f(x),y-x\rangle + \beta_k \sum_{i=1}^ms_i+\frac{\left(\lambda+\beta_k\lambda'\right)}{2}||y-x||^2\\
			\text{s.t.}& f_i(x)+\langle \nabla f_i(x), y-x\rangle \leq s_i,\, i=1 \ldots m\\
 			& s_1,\ldots,s_m \geq 0 \\
 			& y\in Q.
		\end{array}\\
		\text{\bf Step 3 } & \text{ If $\test_i(x_k,x_{k+1})\leq 0$ for all $i=1,\ldots,m$, then $\beta_{k+1}=\beta_k$,}\\
			&  \text{ otherwise }\beta_{k+1}=\beta_k+\delta  
	\end{array}
\end{equation*}
\end{minipage}
}
}

\medskip
The convergence in terms of limit points and for the sequence $\beta_k$ is similar to that of previous section. In this theorem semi-algebraicity is not necessary.
\begin{theorem}
	\label{th:fletcher}
	Assume that the following properties hold
	\begin{itemize}
			\item[--] Lipschitz continuity conditions \eqref{eq:ESQM0}, 
			\item[--] steplength conditions \eqref{fletcher-lip},
		\item[--]  qualification assumption \eqref{eq:ESQM2}, 
		\item[--] boundedness assumptions \eqref{eq:ESQM3}, \eqref{eq:ESQM4},
	\end{itemize}
	then the sequence of parameters $\beta_k$ stabilizes after a finite number of iterations $k_0$ and all cluster points of the sequence $\{x_k\}_{k \in \NN}$ are KKT points of the nonlinear programming problem $\NLP$.
\end{theorem}

We obtain finally the following result:

\begin{theorem}[Convergence of S$\ell^1$QP]
	Assume that the following conditions hold
	\begin{itemize}
	\item[--] The functions $f,f_1,\ldots,f_m$ and the set $Q$ are semi-algebraic,
	\item[--] Lipschitz continuity condition  \eqref{eq:ESQM0}, 
	\item[--] steplength condition \eqref{fletcher-lip},
		\item[--]  qualification assumption  \eqref{eq:ESQM2}, 
		\item[--] boundedness assumptions \eqref{eq:ESQM3}, \eqref{eq:ESQM4}.
		\end{itemize}
 Then, 
 
 \begin{itemize}
 \item[(i)] The sequence $\{x_k\}_{k \in \NN}$ generated by {\rm (ESQM)} converges to a feasible point $x_\infty$ satisfying the KKT conditions for the nonlinear programming problem $\NLP$. 
 \item[(ii)]Either convergence occurs in a finite number of steps or the rate is of the form:
 
(a)  $\|x_k-x_\infty\|=O(q^k)$, with $q\in(0,1)$, 
 
(b)  $\|x_k-x_\infty\|=O\left(\frac{1}{k^{\gamma}}\right)$, with $\gamma>0$. 
\end{itemize}	

\end{theorem}

\section{Majorization-minimization procedures}
\label{sec:mainResult}
\subsection{Sequential model minimization}
We consider  a  general problem of the form
\begin{equation}
\PP	\qquad \min \Big\{f(x):x\in \DD\Big\}
\end{equation}
where $f:\R^n\to \R$ is a  continuous function and $\mathcal{D}$ is a nonempty closed set. 

In what follows we study the properties of majorization-minimization methods. At each feasible point, $x$, local convex models for $f$ and $\DD$ are available,  say $h(x,\cdot):\R^n\to \R$ and $D(x)\subset\RR^n$; we then iteratively solve problems of the form
$$x_{k+1} \in \argmin\Big\{ h(x_k, y):y\in D(x_k)\Big\}.$$

\smallskip

In order to  describe the majorization-minimization method we study, some elementary notions from variational analysis and semi-algebraic geometry are required. However, since the concepts and notations we use are quite standard, we have postponed their formal introduction in Section~\ref{sec:tools} page~\pageref{sec:tools}. We believe this contributes to a smoother presentation of our results.


\subsection{Majorization-minimization procedures 
}\label{sec:assumptions}
For the central problem at stake 
$$\PP\quad \min \Big\{f(x) :x \in \DD\Big\}$$
we make the following standing assumptions 
$$\SSS \quad 
\left\{
\begin{aligned}
 & f:\R^n\rightarrow \R \text{ is locally Lipschitz continuous, subdifferentially regular and semi-algebraic, }\\
 & \inf \Big\{f(x): x\in \DD\Big\}>-\infty,\\
 & \DD\subset\R^n \text{ is nonempty, closed, regular and semi-algebraic}.
\end{aligned}
\right.
$$
\begin{remark}[Role of regularity]{\rm The meaning of the terms subdifferential regularity/regularity is recalled in the next section. It is important to mention that these two assumptions are only used to guarantee the good behavior of the sum rule (and thus of KKT conditions)$$\partial \left(f+i_\DD\right)(x)=\partial f(x)+N_\DD(x),\:x\in \DD.$$ One could thus use   alternative sets of assumptions, like: $f$ is $C^1$ and $D$ is closed (not necessarily regular).}
\end{remark}

A  critical point $x\in \R^n$ for $\PP$ is characterized by the relation $\partial (f+i_\DD)(x)\ni 0$, i.e. using the sum rule:
$$\partial f(x)+N_\DD(x)\ni 0 \quad  \text{\rm (Fermat's rule for constrained optimization)}.$$
 When $\DD$ is a nonempty intersection of sublevel sets, as in Section~\ref{s:nlp}, it necessarily satisfies the assumptions $\SSS$ (see Appendix). Besides, by using the generalized Mangasarian-Fromovitz qualification condition  at $x$,  one sees that Fermat's rule exactly amounts to KKT conditions (see Proposition \ref{crit}).
\subsubsection*{Inner convex constraints approximation}

Constraints are locally modeled at a point $x\in \R^n$ by a  subset  $D(x)$ of $\R^n$. One assumes  that $D:\R^n\rightrightarrows\R^n$ satisfies 
\begin{equation} \label{cont}
\left\{\begin{array}{ll}
 & \dom D\supset\DD,\\
 & D(x)\subset \DD\text{ and } N_{D(x)}\,(x)\subset N_\DD(x), \forall x\in \DD,\\
    & D \text{ has closed convex values},\\
  & D\text{ is continuous (in the sense of multivalued mappings)}.\\

 \end{array}\right.
 \end{equation}

\subsubsection*{Local strongly convex upper models for $f$}
Fix $\mu>0$. 
\begin{equation}\left\{
\begin{aligned}
& & \begin{minipage}{13cm}
The family of local models $$h: \left\{\begin{array}{cll} \R^n\times \R^n & \longrightarrow & \R\\ 
 (x,\;y) &  \longrightarrow  & h(x,y)
\end{array}\right.$$\text{satisfies:}
\end{minipage}   \\  
& & \label{majmin}\\ 
\hspace{-2cm}  & & \begin{minipage}{13cm}
 \begin{itemize}
	\item[\em (i)] $h(x, x) = f(x)$ for all $x$ in $\cal D$,
	\item[\em (ii)] $\partial_y h(x,y)|_{y = x} \subset \partial f(x)$ for all $x$ in $\DD$,
	\item[\em (iii)] For all $x$ in $\DD$, $h(x, y) \geq f(y), \forall y \in D(x)$,
	\item[\em (iv)] $h$ is continuous. For each fixed $x$ in $\DD$, the function $h(x,\cdot)$ is $\mu$ strongly convex.
	\end{itemize}
	\end{minipage} 
\end{aligned}
\right.
\end{equation}


\begin{example}\label{ex:grad}{\rm (a) A typical, but important, example of upper approximations  that satisfy these properties comes from the descent lemma (Lemma~\ref{lem:descent}). Given a $C^1$ function $f$ with  $L_f$-Lipschitz continuous gradient we set $$h(x,y)=f(x)+\langle \nabla f(x),y-x\rangle +\frac{L_f}{2}\|x-y\|^2.$$
Then $h$ satisfies all of the above (with $\DD=D(x)=\R^n$ for all $x$).\\
(b)  SQP methods of the previous section provide more complex examples.   
}
\end{example}


\subsubsection*{A qualification condition for the surrogate problem}
We require a relation between the minimizers of $y \to h(x,y)$ and the general variations of $h$.
Set $$\hat h(x,y)=h(x,y)+i_{D(x)}(y),$$
for all $x,y$ in $\R^n$,  and  $\hat h(x,y)=+\infty$ whenever $D(x)$ is empty.
\begin{itemize}
	\item[] For any compact $C\subset \R^n$, there is a constant $K(C)$, such that,  
	\begin{equation}\label{qual}
	x \in \DD\cap C, \: y \in D(x)\cap C\mbox{ and }(v, 0) \in \partial \hat h(x,y)  \Longrightarrow ||v|| \leq K(C)||x - y||.
	\end{equation}
\end{itemize}


\subsubsection*{The iteration mapping and the value function}
For any fixed $x$ in $\DD$, we define the iteration mapping as the solution of the sub-problem
\begin{equation}\label{subprob}
\Big(\mathcal{P}(x)\Big) \quad \min \Big\{ h(x,y): y\in D(x) \Big\} 
\end{equation}
that is 
\begin{equation}\label{itmapping}p(x) := {\argmin}\Big\{ h(x, y): y\in D(x) \Big\}.\end{equation}
We set for all $x$ in $\DD$, 
\begin{equation}\label{value}
\val (x) =\text{value of $\PPx$}=h(x,p(x)),
\end{equation}
and $\val (x)=+\infty$ otherwise.
\begin{remark}{\rm (a) The restriction ``$x$ belongs to $\DD$" is due to the fact that our process is based on upper approximations and  thus  it is a {\em feasible model}  (i.e. generating sequences in~$\DD$). Note however that this does not mean that non feasible methods cannot be studied with this process (see ESQP and S$\ell^1$QP in the previous section) .\\
(b) Recalling Example \ref{ex:grad}, assuming further that $D(x)=\DD$ for all $x$, and denoting by $P_\DD$ the projection onto $\DD$, the above writes
$$p(x)=P_\DD\left(x-\frac{1}{L_f}\nabla f(x)\right).$$
With these simple instances for $h$ and $D$, we recover the gradient projection iteration mapping.
Note also that for this example $\partial \hat h(x,y)=(v,0) $ implies that 
$$v=(LI_n-\nabla^2 f(x))(x-p(x)).$$
Thus the qualification assumption is trivially satisfied whenever $f$ is $C^2$.
}
\end{remark}

\bigskip

Our general procedure  can be summarized as: \\

\fbox{
\begin{minipage}{13cm}
\begin{center}{\bf Majorization-minimization procedure (MMP)}\end{center}
Assume the local approximations satisfy the assumptions: \begin{itemize}
\item inner constraints approximation \eqref{cont}, 
\item upper objective approximation \eqref{majmin}, 
\item qualification conditions \eqref{qual},
\end{itemize} 
Let $x_0$ be in $\DD$ and define iteratively  
$$x_{k+1} = p(x_k),$$
where $p$ is the iteration mapping \eqref{itmapping}.
\end{minipage}
}

\medskip

\begin{example}{\rm Coming back to our model example, (MMP) reduces simply to the gradient projection method
$$x_{k+1}=P_\DD\left(x_k-\frac{1}{L_f}\nabla f(x_k)\right), \: x_0\in \DD.$$}
\end{example}

\subsection{Main convergence result}


Recall  the standing assumptions $\SSS$ on $\PP$, our main  ``abstract" contribution is the following theorem. 

\begin{theorem}[Convergence of MMP for \-semi-algebraic problems]
	\label{th:convergence}

	Assume that the local model pair $\Big(h,D(\cdot)\Big)$ satisfies: \begin{itemize} 
	\item[--] the inner convex constraints assumptions \eqref{cont}, 
	\item[--]  the upper local model assumptions \eqref{majmin}, 
	\item[--] the qualification assumptions \eqref{qual},
	\item[--] the tameness assumptions: $f,h$ and $D$ are real semi-algebraic.
	\end{itemize} 
	
	Let $x_0 \in \DD$ be a feasible starting point and consider the sequence $\{x_k\}_{k =1, 2,\ldots}$  defined by $x_{k+1} = p(x_k)$.  Then,
	
	\begin{itemize}
	\item[(I)] The following asymptotic alternative holds 
		\begin{itemize}
		\item[(i)] either the sequence $\{x_k\}_{k=1, 2,\ldots}$ diverges, i.e. $\|x_k\|\to+\infty$,
		\item[(ii)] or it converges to a single point $x_\infty$ such that
		$$\partial f(x_\infty)+N_\DD(x_\infty)\ni 0.$$
	\end{itemize}
	\item[(II)] In addition, when $x_k$ converges, either it converges in a finite number of steps or the rate of convergence is of the form:
 
(a)  $\|x_k-x_\infty\|=O(q^k)$, with $q\in(0,1)$, 
 
(b)  $\|x_k-x_\infty\|=O\left(\frac{1}{k^{\gamma}}\right)$, with $\gamma>0$. 
\end{itemize}

		\end{theorem}
\begin{remark}[Coercivity/Divergence]{\rm (a) If in addition  $[f\leq f(x_0)]\cap D$ is bounded, the sequence $x_k$ cannot diverge and converges thus to a critical point.\\
(b) The divergence property $(I)-(i)$ is a positive result, a {\em convergence result}, which does not correspond to a failure of the method but rather {\em to the absence of minimizers in a given zone}.}
\end{remark}

\medskip

Theorem \ref{th:convergence} draws  its strength from the fact that majorization-minimization schemes are ubiquitous in continuous optimization (see Beck and Teboulle (2010) \cite{beck2010gradient}). This is illustrated with SQP methods but other applications can be considered. 

\bigskip

The proof (to be developed in the next section) is not trivial but the ideas can be briefly sketched as follows: 
\begin{itemize}
	\item Study the auxiliary function, the ``value improvement  function":	$$F=\val : \left\{
	\begin{aligned}
	\DD & \to & & \R\\
	x &  \to  & & h(x, p(x)).
    \end{aligned}\right.
	$$ 
	\item Show that there is a non-negative constants $K_1$ such that sequence of iterates satisfies:
		\begin{align*}
			&F(x_k) + K_1 ||x_k - x_{k+1}||^2 \leq f(x_k) \leq F(x_{k-1})
		\end{align*}
	\item Show that for any compact $C$, there is a constant $K_2(C)$ such that if $x_k \in C$, we have:
		\begin{align*}
			&\dist \big(0, \partial F(x_{k})\big) \leq  K_2(C) ||x_{k+1} - x_k||.
		\end{align*}
	\item Despite the {\em explicit type} of the second inequality, one may use KL property (see Section~\ref{sec:tools}) and techniques akin to those presented in Bolte et al. (2013) \cite{bolte2013proximal}, Attouch et al. (2013) \cite{attouch2013convergence} to obtain convergence of the iterative process.
\end{itemize}

\section{Convergence analysis of majorization-minimization procedures}
\label{sec:convergence}
This section is entirely devoted to the exposition of the technical details related to the proof of Theorem~\ref{th:convergence}.
\subsection{Some concepts for nonsmooth and semi-algebraic optimization}
\label{sec:tools}
We hereby recall a few definitions and concepts that structure our main results. In particular, we introduce the notion of a subdifferential and of a KL function, which are the most crucial tools used in our analysis.

\subsubsection{Nonsmooth functions and subdifferentiation}\label{s:nonsmooth}
 A detailed exposition of these notions can be found in Rockafellar and Wets (1998) \cite{rockafellar1998variational}. In what follows, $g$ denotes a proper lower semi-continuous function from $\RR^n$ to $(-\infty, +\infty]$ whose domain is denoted and defined by $\dom g=\big\{x\in \R^n:g(x)<+\infty\big\}$. Recall that $g$ is called proper if $\dom g\neq \emptyset$. 
\begin{definition}[Subdifferentials]{\rm Let $g$ be a proper lower semi\-con\-tinuous func\-tion from $\RR^n$ to $(-\infty, +\infty]$.
	\begin{enumerate}
		\item Let $x \in \dom g$, the {\em Fr\'{e}chet subdifferential of $g$ at $x$} is the subset of vectors $v$ in $\R^n$ that satisfy
			$${\lim\inf}_{y \to x, \: y \neq x}\: \frac{g(y) - g(x) - \left\langle v, y - x \right\rangle}{||x - y||} \geq 0.$$
			When $x \not\in \dom g$, the Fr\'{e}chet subdifferential is empty by definition. The Fr\'{e}chet subdifferential of $g$ at $x$ is denoted by $\hat{\partial} g(x)$.
		\item The {\em limiting subdifferential}, or simply {\em the subdifferential of $g$ at $x$}, is defined by the following closure process:
			$$\partial g(x) = \{v \in \RR^n:\; \exists x_j \to x, g(x_j) \to g(x), u_k \in \hat{\partial} g(x_j), u_j \to v \text{ as } j \to \infty\}.$$
		\item Assume $g$ is finite valued and locally Lipschitz continuous. The function $g$ is said to be {\em subdifferentially regular}, if $\hat{\partial} g(x) = \partial g(x)$ for all $x$ in $\R^n$.
	\end{enumerate}}
\end{definition}

\noindent
Being given a closed subset $C$ of $\RR^n$, its {\em indicator function} $i_C:\R^n\rightarrow(-\infty,+\infty]$ is defined as follows
$$i_C(x)=0 \text{ if }x\in C, \: i_C(x)=+\infty\text{ otherwise}.$$
$C$ is said to be {\em regular} if $\hat{\partial} i_C(x) = \partial i_C(x)$ on $C$. In this case, the {\em normal cone to $C$} is defined by the identity
$$N_C(x)=\partial i_C(x), \forall x \in \RR^n.$$
The {\em distance function to $C$} is defined as 
$$\dist (x,C)=\min \big\{\|x-y\|:y\in C\big\}.$$

\medskip
We recall the two following fundamental results.

\begin{proposition}[Fermat's rule,  critical points, KKT points]\label{crit} We have the following extensions of the classical Fermat's rule:
	\begin{enumerate}
		\item[(i)] When $x$ is a local minimizer of $g$, then $0 \in \partial g(x)$.
		\item[(ii)] If $x$ is a local minimizer of $\PP$, under assumption $\SSS$, then:
 			$$\partial f(x)+N_\DD(x)\ni 0.$$
		\item[(iii)] Assume further that  $\DD$ is of the form
 			$$\DD=\{x\in Q: f_1(x)\leq 0,\ldots,f_m(x)\leq 0\},$$
 			where $Q$ is closed, convex and nonempty and  $f_1,\ldots,f_m:\R^n\to \R$ are $C^1$ functions. For  $x$ in $\DD$, set $I(x)=\{i:f_i(x)=0\}$ and assume that there exists $y\in Q$ such that, \\
 			$$\text{\rm (Robinson QC)}\qquad \langle\nabla f_i(x),y - x\rangle<0, \forall i\in I(x).\qquad \qquad$$\\
			Then $\DD$ is regular, 
 			$$N_\DD(x)=\left\{\sum_{i\in I(x)} \lambda_i \nabla f_i(x): \lambda_i\geq 0, i\in I(x)\right\}+N_Q(x),$$
			and critical points for $\PP$ are exactly KKT points of $\PP$.
	\end{enumerate}
\end{proposition}
\proof (i) is Rockafellar and Wets (1998) \cite[Theorem 10.1]{rockafellar1998variational}). (ii) is obtained by using the sum rule Rockafellar and Wets (1998) \cite[Corollary 10.9]{rockafellar1998variational}. For (iii), regularity and normal cone expression follow from Rockafellar and Wets (1998) \cite[Theorem 6.14]{rockafellar1998variational} (Robinson condition appears there in a generalized form). 
\endproof

Recall that a convex cone $L\subset\R_+^n$ is a nonempty convex set such that $\R_+L\subset L$. Being given a subset $S$ of $\R^n$, the {\em conic hull of  $S$}, denoted $\cone S$ is defined as the smallest convex cone containing $S$. Since a cone always contains $0$, $\cone \emptyset=\{0\}$.

\begin{proposition}[Subdifferential of set-parameterized indicator functions]\label{subdiff}
Let $n_1$, $n_2$, $m$ be positive integers and $g_1,\ldots,g_m:\R^{n_1}\times \R^{n_2}\rightarrow \R$  continuously differentiable functions. Set 
$$C(x)=\left\{ y\in \R^{n_2}:g_i(x,y)\leq 0, \: \: \forall i=1,\ldots,m  \right\}\subset \R^{n_2}, \quad \forall x \in \RR^{n_1},$$
and for any $y\in C(x)$ put $I(x,y)=\{i=1,\ldots,m : g_i(x,y)=0\},$ the set $I(x,y)$ is empty otherwise.
Assume that the following  parametric Mangasarian-Fromovitz qualification condition holds:
 $$\forall (x,y)\in \R^{n_1}\times \R^{n_2},\: \exists d \in \R^{n_1}\times \R^{n_2}, \:\langle\nabla g_i(x,y),d\rangle<0, \forall i\in I(x,y).$$
Consider the real extended-valued function $H:\R^{n_1}\times \R^{n_2}\rightarrow (-\infty,+\infty]$ defined through
$$H(x,y)=\left\{
\begin{aligned}
& i_{C(x)}(y) \text{ whenever $C(x)$ is nonempty}\\
&\\
& +\infty \text{ otherwise.}
\end{aligned}
\right.$$
 Then the subdifferential of $H$ is given by
\begin{equation}\label{soudiff}
\partial H(x,y)=\mbox{\rm cone} \big\{\nabla g_i(x,y): i\in I(x,y) \big\}. 
\end{equation}
\end{proposition}
\begin{proof} For any $(x,y)$ in $\dom H$, set  $G(x,y)=(g_1(x,y),\ldots,g_m(x,y))$.
	Then $H(x,y)=i_{\R_-^m}(G(x,y))$ and $H$ is the indicator of the set $C = \{(x,y) \in \RR^{n_1} \times \RR^{n_2}: \: G(x,y) \in \R_-^m\}$. We will justify the application of the last equality of Rockafellar and Wets (1998) \cite[Theorem 6.14]{rockafellar1998variational}. We fix $(x,y)$ in $\dom H$ such that $G(x,y)\leq 0$  and we set $I= I(x,y)$.  The  abstract qualification constraint required in Rockafellar and Wets (1998) \cite[Theorem 6.14]{rockafellar1998variational}  is equivalent to $\lambda_i \geq 0, \sum_{i \in I} \lambda_i \nabla g_i(x, y) = 0 \Rightarrow \lambda_i = 0$. Using Hahn-Banach  separation  theorem this appears to be equivalent to the parametric MFQC condition. The set $\RR^m_-$ is regular and we can apply Rockafellar and Wets (1998) \cite[Theorem 6.14]{rockafellar1998variational} which assesses that $C$ is regular at $(x,y)$. In this case, the normal cone of $C$ and the sub-differential of $H$ coincide and are given by
	$$\partial H(x,y)= \left\{ \sum_{i=1}^m \lambda_i\nabla g_i(x,y): \lambda \in N_{\R_-^m}(G(x,y)) \right\}=\left\{ \sum_{i\in I} \lambda_i\nabla g_i(x,y): \lambda_i\geq 0\right\}.$$

\end{proof}

\subsubsection{Multivalued mappings}

 A {\em multivalued mapping $F: \R^n\rightrightarrows \R^m$} maps a point $x$ in $\R^n$ to a subset $F(x)$ of $\R^m$. The set
 $$\dom F:=\big\{x\in \R^n: F(x)\neq \emptyset \big\}$$
 is called the {\em domain} of $F$. For instance the subdifferential of a lsc function defines a multivalued mapping
 $\partial f:\R^n\rightrightarrows\R^n$.

Several regularity properties for such mappings are useful in optimization; we focus here on one of the most natural concept: set-valued continuity (see e.g. Dontchev and Rockafellar (2009) \cite[Section 3.B, p. 142]{dontrock}).
\begin{definition}[Continuity of point-to-set mappings]{\rm 
Let $F:\R^n\rightrightarrows \R^m$ and $x$ in $\dom F$.\\
(i) $F$ is called {\em outer semi-continuous at $x$}, if for each sequence $x_j\to x$ and each sequence $y_j\to y$ with $y_j\in F(x_j)$, we have 
$y\in F(x)$.\\
(ii) $F$ is called {\em inner semi-continuous at $x$}, if for all $x_j\to x$ and $y\in F(x)$ there exists a sequence $y_j\in F(x_j)$, after a given term, such that $y_j\to y$.\\
(iii) $F$ is called {\em continuous at $x$} if it is both outer and inner semi-continuous.}
\end{definition}

\subsubsection{The  KL property and some  facts from real semi-algebraic geometry}
KL is a shorthand here for Kurdyka-\L{}ojasiewicz. This property constitutes a crucial tool in our convergence analysis. We consider the nonsmooth version of this property which is given in Bolte et al. (2007) \cite[Theorem 11]{bolte2007clarke} -- precisions regarding concavity of the desingularizing function are given in Attouch et al. (2010) \cite[Theorem 14]{attouch2010proximal}.

\smallskip
 Being given real numbers $a$ and $b$, we set $[a \leq g \leq b] = \{x \in \RR^n:\; a \leq g(x) \leq b\}.$ The sets $[a < g < b]$, $[g<a]$... are defined similarly.

\smallskip

For $\alpha \in (0, +\infty]$, we denote by $\Phi_\alpha$ the class of functions $\varphi:[0,\alpha)\to\R$ that satisfy the following conditions
\begin{itemize}
     \item[(a)] $\varphi(0) = 0$;
	\item[(b)] $\varphi$ is positive, concave and continuous;
	\item[(c)] $\varphi$ is continuously differentiable on $(0, \alpha)$, with $\varphi' > 0$.
\end{itemize}
\begin{definition}[KL property]{\rm
	Let $g$ be a proper lower semi-continuous function from $\RR^n$ to $(-\infty, +\infty]$.
	\begin{enumerate}
	\item[(i)] The function $g$ is said to have the {\em Kurdyka-\L{}ojaziewicz (KL) property} at $\bar{x} \in \dom \partial g$, if there exist $\alpha \in (0, +\infty]$, a neighborhood $V$ of $\bar{x}$ and a function $\varphi \in \Phi_\alpha$ such that
		\begin{equation}\label{loja}\varphi'(g(x) - g(\bar{x})) \, \dist (0, \partial g(x)) \geq 1\end{equation}
		for all $x \in V \cap [g(\bar{x}) < g(x) < \alpha]$.
	\item[(ii)] The function $g$ is said to be a {\em KL function} if it has the KL property at each point of $\dom\partial g$.
	\end{enumerate}}
\end{definition}
KL property basically asserts that a function can be made sharp by a reparameterization of its values. This appears clearly when  $g$ is differentiable and $g(\bar{x}) = 0$, since in this case \eqref{loja} writes:
$$\|\nabla \big(\varphi\circ g\big)(x)\|\geq 1, \quad \forall x \in V \cap [0 < g(x) < \alpha].$$
The function $\varphi$ used in this parameterization is  called a {\em desingularizing function}. As we shall see such functions are ubiquitous in practice, see Attouch et al. (2010) \cite{attouch2010proximal}, Attouch et al. (2013) \cite{attouch2013convergence}.

When $\varphi$ is of the form $\varphi(s)=cs^{1-\theta}$ with $c>0$ and $\theta\in [0,1)$, the number $\theta$ is called {\em a \L ojasiewicz exponent. }
\begin{definition}[Semi-algebraic sets and functions]{\rm 
	\text{ }
	\begin{itemize}
		\item[(i)] A set $A \subset \RR^n$ is said to be semi-algebraic if there exist a finite number of real polynomial functions $g_{ij}, h_{ij}\colon \RR^n \to \RR$ such that
			$$A = \bigcup_{i=1}^p \bigcap_{j = 1}^q \left\lbrace y \in \RR^n:\; g_{ij}(y) = 0, h_{ij}(y) > 0  \right\rbrace$$
		\item[(ii)] A  mapping $G: \RR^n \rightrightarrows \RR^m$ is a said to be semi-algebraic if its graph
			$$\mbox{\rm graph}\,G=\Big\{ (x,y) \in \RR^{n+m}:\; y\in G(x) \Big\}$$
			is a semi-algebraic subset of $\RR^{n+m}$.
			
			Similarly, a real extended-valued function $g: \RR^n \rightarrow (-\infty,+\infty]$ is semi-algebraic if its graph $\big\{ (x,y) \in \RR^{n+1}:\; y=g(x) \big\}$ is semi-algebraic.	\end{itemize}}
\end{definition}

For this class of functions, we have the following result which provides a vast field of applications for our method -- see also Section~\ref{sec:representable}.
\begin{theorem}[Bolte et al. (2007) \cite{bolte2007clarke}, Bolte et al. (2007) \cite{bolte2007lojasiewicz}]\label{th:KL}
Let $g$ be a proper lower semi-continuous function from $\RR^n$ to $(-\infty, +\infty]$. If $g$ is semi-algebraic, then $g$ is a KL function.
\end{theorem}

\subsection{An auxiliary Lyapunov function: the value  function}
\subsubsection*{Basic estimations} \begin{lemma}[Descent lemma]
\label{lem:descent}
Let $g:\R^n\to \R$ be a differentiable function with $L$-Lipschitz continuous gradient. Then for all $x$ and $y$ in $\RR^n$,
$$|g(y) - g(x) - \langle \nabla g(x),y - x\rangle| \leq \frac{L}{2} ||x - y||^2$$
\end{lemma}
The proof is elementary, see e.g. Nesterov (2004) \cite[Lemma 1.2.3]{nesterov2004introductory}.

\begin{lemma}[Quadratic growth of the local models] 
\label{lem:strConv}
$$\text{Fix $x$ in $\DD$. }\text{Then: }\quad h(x,y) - h(x, p(x)) \geq \frac{\mu}{2}||y - p(x)||^2, \quad \forall y \in D(x).$$
\end{lemma}
\begin{proof}
Since $y \to h(x,y)$ is $\mu$ strongly convex, the function
$$D(x)\ni y \to h(x,y) - \frac{\mu}{2}||y - p(x)||^2$$
is convex. Since $p(x)$ minimizes $y \to h(x,y)$ over $D(x)$, it also minimizes $y \to h(x,y) - \frac{\mu}{2}||y - p(x)||^2$. This follows by writing down the first order optimality condition for $p(x)$ (convex setting) and by using the convexity of $y \to h(x,y) - \frac{\mu}{2}||y - p(x)||^2$. The inequality follows readily.
\end{proof}

\begin{lemma}[Descent property]For all $x$ in $\DD$, 
\begin{align}
\label{eq_link}
f(x) &= h(x,x) \geq h(x,p(x))+\frac{\mu}{2} ||x - p(x)||^2 \geq f(p(x)) + \frac{\mu}{2} ||x - p(x)||^2.
\end{align}
\end{lemma}
\begin{proof}
From Lemma~\ref{lem:strConv}, we have for all $x$ in $\DD$ that
$$h(x,x) - h(x,p(x)) \geq \frac{\mu}{2} ||x - p(x)||^2.$$
Therefore from the fact that $h(x,\cdot)$ is an upper model for $f$ on $D(x)$, we infer
\begin{align}
f(x) &= h(x,x) \geq h(x,p(x))+ \frac{\mu}{2} ||x - p(x)||^2  \geq f(p(x)) + \frac{\mu}{2} ||x - p(x)||^2.
\end{align}
\end{proof}

\subsubsection*{Iteration mapping} For any fixed $x$ in $\DD$, we recall that the iteration mapping is defined through:
$$p(x) := {\argmin}\big\{h(x, y):y\in D(x)\}.$$

\begin{lemma}[Continuity of the iteration mapping]
The iteration function 	$p$ is continuous (on $\DD$).
	\label{lem:continuity}
\end{lemma} 

\begin{proof} Let $x$ be a point in $\cal D$ and let $x_j\in \DD$ be a sequence converging to $x$. Fix $y\in D(x)$ and let $y_j$ be a sequence of points such that $y_j\in D(x_j)$ and $y_j\to y$ (use the inner semi-continuity of $D$). We prove first that $p(x_j)$ is bounded.  To this end, observe that
\begin{equation}\label{majj}
h(x_j,p(x_j))+\frac{\mu}{2}\|y_j-p(x_j)\|^2\leq h(x_j,y_j).
\end{equation}
Recall that $h(x_j,p(x_j))\geq f(p(x_j))\geq \inf_\DD f>-\infty.$ Thus 
$$\frac{\mu}{2}\|y_j-p(x_j)\|^2\leq h(x_j,y_j)-\inf_\DD f$$
and $p(x_j)$ is bounded by continuity of $h$.
Denote by $\pi$ a cluster point of $p(x_j)$. Observe that since $p(x_j)\in D(x_j)$,  the outer semi-continuity of $D$ implies that $\pi\in D(x)$. Passing to the limit in \eqref{majj} above, one obtains 
$$
h(x,\pi)+\frac{\mu}{2}\|y-\pi\|^2\leq h(x,y).
$$
Since this holds for  arbitrary $y$ in $D(x)$, we have established that $\pi$ minimizes $h(x,\cdot)$ on $D(x)$, that is 
$\pi=p(x)$. This proves that $p$ is continuous.\end{proof}

\begin{lemma}[Fixed points of the iteration mapping] Let $x$ be in $\DD$ such that $p(x)=x$. Then $x$ is critical for $\PP$ that is 
$$\partial f(x)+N_\DD(x)\ni0.$$
\end{lemma}
\begin{proof} Using the optimality condition and the sum rule for subdifferential of convex functions one has
\begin{equation}\label{optmodel}
\partial_yh(x,p(x))+N_{D(x)}\ni 0.
\end{equation}
By assumption \eqref{majmin} (ii), we have $\partial_yh(x,p(x))=\partial_yh(x,x)\subset \partial f(x)$. On the other hand $D(x)\subset \DD$ and $N_{D(x)}\,(x)\subset N_\DD(x)$, by \eqref{cont}. Using these inclusions in \eqref{optmodel} yields the result.

\end{proof}

\subsubsection*{Value  function}
The value function is defined through
$$
F=\val:\left\{\begin{array}{lll}\R^n& \longrightarrow&  (-\infty,+\infty] \\
& \\
x & \longrightarrow & h\big(x,p(x)\big).
\end{array}\right.
$$
 Being given $x$ in $\R^n$, and a value $f(x)=h(x,x)$, it measures the progress made not on the objective $f$, but on the value of the model. 

Tarski-Seidenberg theorem asserts a linear projection of a semi-algebraic set is semi-algebraic set. This implies that the class of semi-algebraic functions is closed under many operations, such as addition, multiplication, composition, inverse, projection and partial minimization (see  Bochnak et al. (2003) \cite{coste} and Attouch et al. (2013) \cite[Theorem 2.2]{attouch2013convergence} for an illustration in optimization). Applying standard techniques of semi-algebraic geometry, we obtain therefore: 
\begin{lemma}[Semi-algebraicity of the value function]\label{lem:FLocLip}
 If $f$, $h$, ${\mathcal D}$ are semi-al\-gebraic then $F$ is semi-algebraic.
\end{lemma}


Let $\mathcal{D'}$ denote the domain where $F$ is differentiable. By standard stratification results, this set contains a dense {\em finite} union of open sets (a family of strata of maximal dimension, see e.g. Van Den Dries and Milller (1996) \cite[4.8]{Dries-Miller96}, see also Ioffe (2009) \cite[Theorem 2.3]{ioffe2009invitation} for a self contained exposition). Thus we have:
\begin{equation}\label{strat}
\mathcal{\mbox{int}\, D'} \text{ is dense in }\mathcal{D}.
\end{equation} We now have the following estimates

\begin{lemma}[Subgradient bounds]
\label{lem:stepCond}
Let $C \subset \mathcal{D}$ be a bounded set. Then there exists $K\geq 0$ such that $\forall x \in \mathcal{D'}\cap C$
\begin{equation}\label{bound0}
||\nabla F(x)|| \leq K ||p(x) - x||.
\end{equation}
As a consequence 
\begin{equation}\label{bound}
\dist(0,\partial F(x)) \leq K||p(x) - x||, \, \forall x \in \mathcal{D}\cap C.
\end{equation}
\end{lemma}
\begin{proof}
Fix $\bar{x}$ in $\mbox{int}\,\mathcal{D'}\cap C$ and let  $\delta$ and $\mu$ be in $\R^n$. Then
\begin{eqnarray*}
\hat h(\bar{x} + \delta, p(\bar{x}) + \mu)) & = & h(\bar{x} + \delta, p(\bar{x}) + \mu) +i_{D(\bar x+\delta)}(p(\bar x)+\mu)\\
 & \geq &  h(\bar{x} + \delta, p(\bar{x} + \delta))\\
&  = & h(\bar{x}, p(\bar{x})) + \left\langle\nabla F(\bar{x}), \delta \right\rangle+ o(||\delta||)\\
&  = & \hat h(\bar{x}, p(\bar{x})) + \left\langle\nabla F(\bar{x}), \delta \right\rangle+ o(||\delta||).\\
\end{eqnarray*}
This implies that $(\nabla F(\bar{x}), 0) \in \partial \hat h(\bar{x},p(\bar{x}))$. Since $C$ is bounded, the qualification assumption of Section~\ref{sec:assumptions} yields \eqref{bound0}.

To obtain \eqref{bound}, it suffices to use the definition of the subdifferential, the continuity of $p$ and the fact that $\mbox{int}\,\mathcal{D'}$ is dense in $\mathcal{D}$.
\end{proof}




We have the following property for the sequence generated by the method
\begin{proposition}[Hidden gradient steps]	\label{prop:sequence}
Let $\{x_k\}_{k =1, 2,\ldots}$ be the sequence defined through $x_{k+1}=p(x_k)$ with $x_0\in \DD$. Then $x_k$ lies in $\DD$ and 
\begin{align}
F(x_k) + \frac{\mu}{2} ||x_k - x_{k+1}||^2 &\leq f(x_k) \leq F(x_{k-1}), \: \forall k\geq 1.\label{eq:decrease}
\end{align}
Moreover, for all compact subset $C$ of $\R^n$, there exists $K_2(C)>$ such that $$\dist(0, \partial F(x_k)) \leq K_2(C)||x_{k+1} - x_k||, \text{ whenever }x_k \in C.$$
\end{proposition}
\begin{proof}
The sequence $x_k$ lies in $\DD$ since $p(x_k)\in D(x_k)\subset\DD$ by \eqref{cont}. We only need to prove the second item \eqref{eq:decrease} since the third one immediately follows from \eqref{bound}. Using inequality (\ref{eq_link}) and the fact that $h(x,y) \geq f(y)$ for all $y$ in $D(x)$, we have
\begin{eqnarray*}
F(x) &  =  & h(x,p(x))\\
&  \geq & f(p(x))\\
& = & h(p(x),p(x))\\
&  \geq & F(p(x)) + \frac{\mu}{2} ||p(x) - p(p(x))||^2,
\end{eqnarray*}
therefore
$$F(x_{k-1}) \geq f(x_k) \geq F(x_k) + \frac{\mu}{2} ||x_k - x_{k+1}||^2$$
which proves (\ref{eq:decrease}).
\end{proof}

\subsection{An abstract convergence result}
\label{sec:abstrConv}
The following abstract result is similar in spirit to Attouch et al. (2013) \cite{attouch2013convergence} and to recent variations Bolte et al. (2013) \cite{bolte2013proximal}. However,  contrary to previous works it deals with conditions on a triplet $\{x_{k-1},x_k,x_{k+1}\}$ and the subgradient estimate is of explicit type (like in Absil et al. (2005) \cite{absil2005convergence} and even more closely Noll (2014) \cite{noll}). 

\begin{proposition}[Gradient sequences converge]
	\label{prop:abstConv}
Let $\bar{G}\colon \RR^n \to (-\infty,+\infty]$ be proper, lower semi-continuous,
 semi-algebraic function. Suppose that there exists a sequence $\{x_k\}_{k \in \NN}$ such that, 
\begin{itemize}
	\item[\rm (a)] $\exists K_1>0$ such that $\bar{G}(x_{k}) + K_1||x_{k+1} - x_k||^2 \leq \bar{G}(x_{k-1})$
	\item[\rm (b)] For all compact subset $C$ of $\R^n$, there exists $K_2(C)>$ such that $$\dist(0, \partial \bar{G}(x_k)) \leq K_2(C)||x_{k+1} - x_k||, \text{ whenever }x_k \in C.$$
\item[\rm (c)] If  there exists $x_{k_j}\to \bar x$ as $j\to+\infty$, 
	then $\bar{G}(x_{k_j})\to \bar{G}(\bar x)$.
 \end{itemize}
 
 Then, 
\begin{enumerate}
\item[(I)] The following asymptotic alternative holds:
\begin{itemize}
\item[(i)] Either the sequence $\{x_k\}_{k \in \NN}$ satisfies $\|x_k\|\rightarrow +\infty$,
\item[(ii)] or it converges to a critical point of~$\bar{G}$.
\end{itemize}
As a consequence each bounded sequence is a converging sequence.

\item[(II)] When $x_k$ converges,  we denote by  $x_\infty$ its limit and we take $\theta \in [0,1)$ a \L ojasiewicz exponent of $\bar{G}$ at $x_\infty$. Then,

(i) If $\theta=0$, the sequence  $(x_k)_{k\in \mathbb{N}}$ converges in a finite number of steps,

(ii) If $\theta \in (0,\frac{1}{2}]$ then there exist $c>0$ and
 $q\in [0,1)$ such that
$$\|x_k-x_{\infty}\|\leq c\:q^k, \forall k\geq1.$$

(iii) If $\theta \in (\frac{1}{2},1)$ then there exists $c>0$ such that

$$\|x_k-x_{\infty}\| \leq c\:k^{-\frac{1-\theta}{2\theta-1}}, \forall k\geq1.$$

\end{enumerate}
\end{proposition}

\begin{proof} We first deal with $(I)$. Suppose that there exists $k_0\geq 0$ such that $x_{k_0+1} = x_{k_0}$. This implies by~(a), that $x_{k_0+l} = x_{k_0}$ for all $l > 0$. Thus the sequence converges and the second inequality (b) implies that we have a critical point of $\bar{G}$. We now suppose that $||x_{k+1} - x_k|| > 0$ for all~$k\geq 0$.\\
\noindent{\em Definition of a KL neighborhood.} Suppose that $(I)(i)$ does not hold. There exists therefore a cluster point $\bar x$ of $x_k$. Combining (a) and (c), we obtain that
	\begin{equation}\label{const}
		\lim_{k\to+\infty}\bar{G}(x_k)=\bar{G}(\bar x).
	\end{equation}
	With no loss of generality, we assume that $\bar{G}(\bar x)=0$. Since $\bar{G}$ is semi-algebraic, it is a KL function (Theorem~\ref{th:KL}).  There exist $\delta > 0$, $\alpha > 0$ and $\varphi \in \varphi_\alpha$ such that 
	$$\varphi'(\bar{G}(x))\, \dist(0, \partial \bar{G}(x)) \geq 1,$$
	for all $x$ such that $\|x-\bar x\|\leq \delta$ and $x \in [0 < \bar{G} < \alpha]$. In view of assumption (b), set $$K_2=K_2\left(\bar B(\bar x,\delta)\right).$$
	\noindent
	{\em Estimates within the neighborhood.} Let $r\geq s>1$ be some integers and assume that the points $x_{s-1},x_{s}\ldots, x_{r-1}$ belong to $B(\bar x,\delta)$ with $\bar{G}(x_{s-1}) < \alpha$. Take $k \in \{s,\ldots,r \}$, using (a), we have
\begin{align*}
	\bar{G}(x_{k})  &\leq \bar{G}(x_{k-1}) - K_1||x_{k+1} - x_k||^2 \\
	&= \bar{G}(x_{k-1}) - K_1 \frac{||x_{k+1} - x_k||^2}{||x_{k} - x_{k-1}||} ||x_{k} - x_{k-1}|| \\
	&\leq  \bar{G}(x_{k-1}) - \frac{K_1}{K_2} \frac{||x_{k+1} - x_k||^2}{||x_{k} - x_{k-1}||} \dist(0, \partial \bar{G}(x_{k-1})).
\end{align*}
From the monotonicity and concavity of $\varphi$, we derive 
$$\varphi \circ \bar{G}(x_{k}) \leq \varphi \circ \bar{G}(x_{k-1}) - \varphi' \circ \bar{G}(x_{k-1}) \frac{K_1}{K_2} \frac{||x_{k+1} - x_k||^2}{||x_{k} - x_{k-1}||}  \dist(0,\partial \bar{G}(x_{k-1})),$$
thus by using KL property, for  $k \in \{s,\ldots,r \}$,
\begin{equation}
	\varphi \circ \bar{G}(x_{k}) \leq \varphi \circ \bar{G}(x_{k-1}) - \frac{K_1}{K_2} \frac{||x_{k+1} - x_k||^2}{||x_{k} - x_{k-1}||}. \label{eq:ineqSum}
\end{equation}
We now use the following simple fact: for $a > 0$ and $b \in \RR$,
$$2(a - b) - \frac{a^2- b^2}{a} = \frac{a^2 - 2ab + b^2}{a} = \frac{(a - b)^2}{a} \geq 0,$$
thus for $a>0$  and $b \in \RR$
\begin{equation}
2(a - b) \geq \frac{a^2- b^2}{a}. \label{eq:ineqab}
\end{equation}
We have therefore, for $k$ in $\{s,\ldots,r \}$,
\begin{align*}
||x_{k} - x_{k-1}|| &= \frac{||x_{k} - x_{k-1}||^2}{||x_{k} - x_{k-1}||}\\
&= \frac{||x_{k+1} - x_k||^2}{||x_{k} - x_{k-1}||} + \frac{ ||x_{k} - x_{k-1}||^2 - ||x_{k+1} - x_k||^2}{||x_{k} - x_{k-1}||}\\
&\overset{(\ref{eq:ineqab})}{\leq} \frac{||x_{k+1} - x_k||^2}{||x_{k} - x_{k-1}||} + 2(||x_{k} - x_{k-1}|| - ||x_{k+1} - x_k||)\\
&\overset{(\ref{eq:ineqSum})}{\leq}  \frac{K_2}{K_1}\left(\varphi \circ \bar{G}(x_{k-1}) - \varphi \circ \bar{G}(x_k)\right) +  2(||x_{k} - x_{k-1}|| - ||x_{k+1} - x_k||).
\end{align*}
Hence by summation
\begin{equation}\label{cauchy}
	\sum_{k=s}^{r} ||x_{k} - x_{k-1}||\leq \frac{K_2}{K_1}\Big(\varphi \circ \bar{G}(x_{s-1}) - \varphi \circ \bar{G}(x_r)\Big) + 2 \left(||x_{s} - x_{s-1}|| - ||x_{r+1} - x_r||\right).
\end{equation}
{\em The sequence remains in the neighborhood and converges.} Assume that for $N$ sufficiently large one has 
\begin{align}
\|x_N-\bar x\|&\leq \frac{\delta}{4},\label{petit0}\\
\frac{K_2}{K_1}\Big(\varphi \circ \bar{G}\Big)(x_{N}) &\leq \frac{\delta}{4},\label{petit}\\
\sqrt{K_1^{-1}\bar{G}(x_{N-1})}& <  \min\left(\frac{\delta}{4},  \sqrt{K_1^{-1}\alpha}\right)\label{petit1}.
\end{align}
One can require \eqref{petit} and \eqref{petit1} because $\varphi$ is continuous and $\bar{G}(x_k)\downarrow 0$. By (a), one has 
\begin{equation}\label{up}
	\|x_{N+1}-x_N \| \leq \sqrt{K_1^{-1}\bar{G}(x_{N-1})} < \frac{\delta}{4}.
\end{equation}  
Let us prove that $x_r \in B(\bar{x}, \delta)$ for $r \geq N + 1$. We proceed by induction on $r$. By \eqref{petit0}, $x_N \in B(\bar{x}, \delta)$  thus the induction assumption is valid for $r=N+1$.  Since by \eqref{petit1} one has $\bar{G}(x_N) < \alpha$,  estimation \eqref{cauchy} can be applied with $s= N+1$. Suppose that $r \geq N+1$ and $x_{N}, \ldots, x_{r-1} \in B(\bar x,\delta)$, then we have the following 
\begin{eqnarray*}
\|x_{r}-\bar x\| & \leq & \|x_{r}-x_{N}\| + \|x_N-\bar x\|\\
& \overset{\eqref{petit0}}\leq & \sum_{k=N+1}^{r}\|x_{k}-x_{k-1}\|+\frac{\delta}{4}\\
& \overset{\eqref{cauchy}}\leq & \frac{K_2}{K_1}\varphi \circ \bar{G}(x_{N}) + 2 ||x_{N+1} - x_{N}||+ \frac{\delta}{4}\\
& \overset{\eqref{petit}, \eqref{up}}< & \delta.
\end{eqnarray*}
Hence $x_{N}, \ldots, x_{r} \in B(\bar x,\delta)$ and the induction proof is complete. Therefore, $x_r \in B(\bar{x}, \delta)$ for any $r \geq N$. Using \eqref{cauchy} again, we obtain that the series $\sum \|x_{k+1}-x_k\|$ converges, hence $x_k$ also converges by Cauchy criterion.\\

The second part {\em(II)} is proved as in Attouch and Bolte (2009) \cite[Theorem 2]{attouch2009convergence}. First, because of the semi-algebraicity of the data, $\varphi$ can be chosen of the form $\varphi(s)=c.s^{1-\theta}$ with $c>0$ and $\theta\in [0,1)$. In this case, \eqref{cauchy} combined with KL property and (b) yields a similar result as formula (11) in Attouch and Bolte (2009) \cite{attouch2009convergence}, which therefore leads to the same estimates.
 
\end{proof}
\begin{remark}\label{coercive}{\rm (1) {\bf (Coercivity implies convergence)} Quite often in practice $\bar{G}$ has bounded level sets. In that case the alternative reduces to convergence because of assumption (a).\\
	(2) {\bf (Assumption (c))} Assumption (c) is very often satisfied in practice: for instance when $\bar{G}$ has a closed domain and is continuous on its domain or when $\bar{G}$ is locally  convex up to a square (locally semi-convex).}
\end{remark}

At last, Propositions \ref{prop:sequence} and \ref{prop:abstConv} can be combined to prove Theorem \ref{th:convergence}. First, we can consider the restriction of $F$ to the closed semi-algebraic set $\mathcal{D}$, since the sequence of Proposition \ref{prop:sequence} stays in $\mathcal{D}$. $F$ is semi-algebraic by Lemma \ref{lem:FLocLip} and $F$ is continous on $\mathcal{D}$ by continuity of $h$ and $p$. Propositions \ref{prop:sequence} shows that $F$ satisfies assumptions (a) and (b) of Proposition \ref{prop:abstConv}, and assumption (c) follows by the previous remark. Hence Proposition \ref{prop:abstConv} applies to $F$ and the result follows. 

\section{Beyond semi-algebraicity: MMP and NLP with real analytic data}
\label{sec:representable}

Many concrete and essential problems involve objectives and constraints defined through real analytic functions --which are not in general semi-algebraic functions-- and this raises the question of the actual scope of the results described previously. We would thus like to address here the following question: {\em Can we deal with nonlinear programming problems involving real analytic data?} 

\smallskip

A convenient framework to capture most of what is  needed to handle real analytic problems, and of an even larger class of problems, is the use of o-minimal structures. These are classes of sets and functions whose stability properties and topological behavior are the same as those encountered  in the semi-algebraic world.

\bigskip

We give below some elements necessary  to understand what is at stake and how our results enter this larger framework.

\begin{definition}[O-minimal structures, {\rm see Van Den Dries and Miller (1996) {\cite{Dries-Miller96}}}]
\label{d:omin}
{\em
An {\em o-minimal} structure on $(\R,+,.)$ is a sequence of families ${\cal O}=(\mathcal{O}_{p})_{p\in \NN}$ with $\mathcal{O}_{p}\subset{\mathcal P}(\R^p)$ (the collection  of subsets of $\R^p$), such that for each $p\in\NN$:
\begin{enumerate}\itemsep=1mm
\item[(i)]
Each $\mathcal{O}_{p}$ contains $\R^p$ and is stable by finite union, finite intersection and complementation;
\item[(ii)]
if $A$ belongs to $\mathcal{O}_{p}$, then $A\times\R$ and
$\R\times A$ belong to $\mathcal{O}_{p+1}$ ;
\item[(iii)]
if $\Pi:\R^{p+1}\rightarrow\R^p$ is the canonical
projection onto $\R^p$ then for any $A$ in $\mathcal{O} _{p+1}$,
the set $\Pi(A)$ belongs to $\mathcal{O}_{p}$ ;
\item[(iv)]
$\mathcal{O}_{p}$ contains the family of real algebraic
subsets of $\R^p$, that is, every set of the form
\[
\{x\in\R^p:g(x)=0\},
\]
where $g:\R^p\rightarrow\R$ is a real polynomial function ;
\item[(v)]
the elements of $\mathcal{O}_{1}$ are exactly the finite
unions of intervals.
\end{enumerate}}
\end{definition}

Examples of such structures are given in Van Den Dries and Milller (1996) \cite{Dries-Miller96}. We focus here on the class of globally subanalytic sets which allows us to deal with real analytic NLP in a simple manner. Thanks to Gabrielov's theorem of the complement, the {\em class of globally subanalytic subsets} can be seen as the smallest o-minimal structure containing semi-algebraic sets and the graphs of all real analytic functions of the form: $f:[-1,1]^n\to \R$, see e.g. Van Den Dries and Milller (1996) \cite{Dries-Miller96}.  As a consequence any real analytic function defined on an open neighborhood of a box is globally subanalytic. 

\smallskip

Note that a real analytic function might not be globally subanalytic (take $\sin$ whose graph intersects the $x$-axis infinitely many times, and thus (iv) is not fulfilled for $\left(\mbox{graph}\;\sin\right)\,\cap\,(\text{O}x)$, however it follows from the definition that  the restriction of a real analytic function to a compact set  included in its (open) domain is globally subanalytic.

\bigskip

We come now to the results we need for our purpose. For any o-minimal structure, one can assert that: 
\begin{itemize}
\item[(a)] The KL property holds -- i.e. one can replace the term ``semi-algebraic" by ``definable" in Theorem~\ref{th:KL}, see Bolte et al. (2007) \cite{bolte2007clarke}.
\item[(b)] The stratification properties \eqref{strat} used to derive the abstract qualification condition hold, see Van Den Dries and Milller (1996) \cite{Dries-Miller96}.\end{itemize}

\medskip
\fbox{
\begin{minipage}{16cm}
{\em As a consequence, and at the exception of convergence rates, all the results announced in the paper are actually valid for an arbitrary o-minimal structure instead of the specific choice of the class of semi-algebraic  sets. }
\end{minipage}
}

\bigskip
To deal with the case of real analytic problems, we combine the use of compactness and of the properties of globally subanalytic sets. This leads to the following results. 

\begin{theorem}[Convergence of ESQM/S$\ell^1$QP for analytic functions]
	Assume that the following properties hold
	\begin{itemize}
	\item[--] The functions $f,f_1,\ldots, f_m$ are real analytic and $Q$ is globally subanalytic {\rm(\footnote{$Q$ subanalytic is actually enough, see Van Den Dries and Milller (1996) \cite{Dries-Miller96}})},
	\item[--] Lipschitz continuity assumptions \eqref{eq:ESQM0},
	\item[--] steplength condition  \eqref{aus-lip},		
	\item[--]  qualification assumptions \eqref{eq:ESQM2}, 
		\item[--] boundedness assumptions \eqref{eq:ESQM3}, \eqref{eq:ESQM4}.
		\end{itemize}
 Then, \begin{itemize}
 \item[(i)] the sequence $\{x_k\}_{k \in \NN}$ generated by {\rm (ESQM)} (resp. {\rm S$\ell^1$QP}) converges to a feasible point $x_\infty$ satisfying the KKT conditions for the nonlinear programming problem $\NLP$.
 
 \item[(ii)] Either convergence occurs in a finite number of steps or the rate is of the form:
 
(a)  $\|x_k-x_\infty\|=O(q^k)$, with $q\in(0,1)$, 
 
(b)  $\|x_k-x_\infty\|=O\left(\frac{1}{k^{\gamma}}\right)$, with $\gamma>0$. 
\end{itemize}	
		 
\end{theorem}
\begin{theorem}[Convergence of the moving balls method]\label{t:moving}
Recall that $Q=\RR^n$ and assume that the following properties hold
	\begin{itemize}
	\item[--] The functions $f,f_1,\ldots,f_m$ are real analytic,
	\item[--] Lipschitz continuity assumptions \eqref{eq:MB1}, 
	\item[--] Mangasarian-Fromovitz qualification condition \eqref{eq:MB3},
		\item[--] boundedness condition  \eqref{eq:MB2},
		\item[--] feasibility of the starting point $x_0 \in \mathcal{F}$.	
		\end{itemize}

	Then,
	
\begin{itemize}
\item[(i)] The sequence $\{x_k\}_{k \in \NN}$  defined by the moving balls method converges to a feasible point $x_\infty$ satisfying the KKT conditions for the nonlinear programming problem $\NLP$.

\item[(ii)] Either convergence occurs in a finite number of steps or the rate is of the form:
 
(a)  $\|x_k-x_\infty\|=O(q^k)$, with $q\in(0,1)$, 
 
(b)  $\|x_k-x_\infty\|=O\left(\frac{1}{k^{\gamma}}\right)$, with $\gamma>0$. 
\end{itemize}

\end{theorem}

\begin{proof} The ``proofs" of both theorems are the same. We observe first that in both cases the sequences are bounded. Let thus $a>0$ be such that $x_k\in [-a,a]^n$ for all nonnegative $k$.
Now the initial problem can be artificially transformed to a definable problem by including the constraints $x_i\leq a$ and $-x_i\leq a$ without inducing any change for the sequences. This imposes restrictions to real analytic function making them globally subanalytic hence definable.

The fact that the rate of convergence are of the same nature is well known and comes from the fact that Puiseux Lemma holds for subanalytic functions (see Van Den Dries and Milller (1996) \cite[5.2]{Dries-Miller96} and the discussion in Kurdyka (1998) \cite[Theorem \L{}I]{kurdyka1998gradients}).
\end{proof}

\section{Appendix: convergence proofs for SQP methods}
\label{sec:appendix}

\subsection{Convergence of the moving balls method}

The local model of $f$ is given at a feasible $x$ by
$$h_{\rm MB}(x, y) = f(x) + \left\langle \nabla f(x), y - x \right\rangle + \frac{L}{2} ||y - x||^2, \: x,y\in \R^n,$$
while the constraint approximation is given by
$$D(x)=\left\{y\in\RR^n: f_i(x) + \left\langle \nabla f_i(x), y - x \right\rangle + \frac{L_i}{2} ||y - x||^2  \leq 0 \right\}.$$

The fact that for all $x$ in  $\mathcal{F}$, $D(x) \subset \mathcal{F}$  is ensured by Lemma~\ref{lem:descent}.  As an intersection of a finite number of balls containing $x$ the set $D(x)$ is a nonempty compact (hence closed) convex set. The proof of the continuity of $D$ is as in Auslender et al. (2010) \cite[Proposition A1 \& A2]{auslender2010moving}.

Let us also recall that Mangasarian-Fromovitz condition implies that
\begin{lemma}[Slater condition for  $\mathcal{P}(x)$]{Auslender et al. (2010) \cite[Proposition 2.1]{auslender2010moving}}
The set $D(x)$ satisfies the Slater condition for each $x$ in $\FF$.
\end{lemma}

\begin{corollary}
	\label{cor:lambdaZero}
	For a given feasible $x$, set	
		$$g_i(y)=f(x)+\langle\nabla f_i( x), y- x\rangle +\frac{L_i}{2}\|y- x\|^2, \; y\in \R^n,\; i = 1, \ldots, m.$$
		Suppose that $(x,y)$ is such that $g_i(y) \leq 0$, $i = 1 \ldots m$. Then the only solution $u=(u_1,\ldots,u_m)$ to 
		\begin{align*}
			\sum_{i=1}^m u_i\, \nabla g_i(y)&= 0,\, u_i \geq 0 \text{ and } u_i \,g_i(y)=0 \text{ for $1 \leq i \leq m$}
		\end{align*}
		is the trivial solution $u = 0$.
\end{corollary}
\begin{proof}
	When $J = \{i=1,\ldots,m: \; g_i(y) = 0\}$ is empty, the result is trivial. Suppose that $J$ is not empty and argue by contradiction. This means that $0$ is in the convex envelope of $\big\{\nabla g_j(y), j\in J\big\}$ and thus one cannot have Mangasarian-Fromovitz condition for $\mathcal{P}(x)$ at $y$ (recall that $\mathcal{P}(x)$ involves constraints of the form $g_i \leq 0$). This contradicts the fact that Slater condition holds for $\mathcal{P}(x)$, since Slater condition classically implies Mangasarian-Fromovitz condition at each point. 
\end{proof}

\begin{corollary}[Lagrange multipliers of the subproblems are bounded]\label{lagbound}
For each $x$ in $\FF$, we denote by $\Lambda(x)\subset\R_+^m$ the set of Lagrange multipliers associated to $\mathcal{P}(x)$. 
Then for any compact subset $B$ of $\FF$,
\begin{equation}
\sup \left\{\max_{i=1,\ldots,m} \lambda_i(x): (\lambda_1(x),\ldots,\lambda_m(x))\in \Lambda(x), x\in B\right\}<+\infty.
\end{equation}
\end{corollary}
\begin{proof} Observe that, at this stage, we know that $p$ is continuous. 
 We argue by contradiction and assume that the supremum is not finite. 	One can thus assume, by compactness, that there exists a point $\bar x$ in $\FF$ together with a sequence $z_j\to \bar x$ such that at least one of the $\lambda_j(z_j)$  tends to infinity. Writing down the optimality conditions, one derives Lagrange relations
 	$$ \frac{1}{\sum_{i=1}^m  \lambda_i(x)}\left(\nabla f(z_j) + L(p(z_j) - z_j)\right) +\sum_{i=1}^m  \frac{\lambda_i(z_j)}{\sum_{i=1}^m  \lambda_i(z_j)}\left( \nabla f_i(z_j) + L_i (p(z_j) - z_j)\right) = 0$$
	and complementary slackness
	$$\lambda_{j}(z_j)\left(f_i(z_j)+\langle\nabla f_i(z_j),p(z_j)-z_j)+\frac{L_i}{2}\|p(z_j)-z_j\|^2\right)=0.$$
		Up to an extraction one can assume that the sequence of $p$-uplet $\Big\{\Big (\frac{\lambda_i(z_j)}{\sum_{i=1}^m  \lambda_i(z_j)}\Big)_{i=1,\ldots,m}\Big\}_j$ converges to $u=(u_1,\ldots,u_m)$ in the unit simplex and that, for all $i$,  the limit of $\lambda_i(z_j)$ exists and is either finite or infinite. 
		Passing to the limit, one obtains that 
		\begin{align}
			\label{eq:zeroSumSlater}
		\sum_{i = 1}^m u_i\, \nabla g_i(p(\bar x))&= 0,\, g_i(p(\bar x))\leq 0 \text{ and } u_i \,g_i(p(\bar x))=0 \text{ for $1 \leq i \leq m,$ }
		\end{align}
		where $g_1,\ldots,g_m$ are as defined in Lemma \ref{cor:lambdaZero}. But Lemma \ref{cor:lambdaZero} asserts that the unique solution to such a set of equations is $u = 0$, which contradicts the fact that $u$ is a point of the unit simplex. 
\end{proof}

Recall that for all $x,y\in \R^n$, we set $\hat h_{\rm MB}(x,y)=h_{\rm MB}(x,y)+i_{D(x)}(y)$. Fix  $x\in \FF$ and $y$ in $D(x)$  set 
$$I(x, y) = \left\{i\in\{1,\ldots,m\} : \; f_i(x) + \left\langle \nabla f_i(x), y - x \right\rangle + \frac{L_i}{2} ||y - x||^2 = 0\right\}.$$
Combining Proposition~\ref{subdiff} with Corollary~\ref{cor:lambdaZero}, one has that the subdifferential of $\hat h_{\rm MB}$  is given by
\begin{align} \label{mov}
	 & \partial \hat h_{\rm MB}(x, y)\\	
	= &\left( 
	\begin{array}{c}
		L(x - y) - \nabla^2 f(x) (x - y)\\
		\nabla f(x) + L(y - x)
	\end{array}
	\right) + 
	\cone\left\lbrace\left(  
	\begin{array}{c}
		L _i(x - y) - \nabla^2 f_i(x) (x - y)\\
	\nabla f_i(x) + L_i (y - x)
	\end{array}
	\right),\; i \in I(x, y)\right\rbrace. \nonumber
\end{align}
The only assumption of Section~\ref{sec:assumptions} that needs to remain established is the qualification assumption~\eqref{qual}. 
\begin{lemma}
	\label{lem:qualif}
	The qualification assumption \eqref{qual}  holds for $h_{\mathrm{MB}}$.
\end{lemma}

\begin{proof}
	$(v, 0) \in \partial \hat{h}_{\mathrm{MB}}(x, y)$ implies that 
	\[
		y = \argmin_z \{\hat{h}_{\mathrm{MB}} (x, z):z\in D(x)\},
	\]
	in other words that $y= p(x)$. In view of \eqref{mov}, one has the existence of non-negative $\lambda_i(x), i=1,\ldots,m$ such that
\begin{eqnarray}
\Big(L .I_n - \nabla^2 f(x)+\sum_{i=1}^p \lambda_i(x)\left( L _i I_n - \nabla^2 f_i(x)\right)\Big)(x-p(x))& = v,\nonumber\\ 
	\nabla f(x) + L(y - x)+\sum_{i=1}^p  \lambda_i(x)\left( \nabla f_i(x) + L_i (p(x) - x)\right) & = 0. \label{mult}
\end{eqnarray}
	
	The desired bound on $v$ follows from the bound on the the Lagrange multipliers in \eqref{mult} obtained in Corollary~\ref{lagbound}
\end{proof}

\bigskip

The assumptions for applying  Theorem \ref{th:convergence} are now gathered and Theorem~\ref{t:moving} follows.  The fact that we eventually obtain a KKT point is a consequence of the qualification condition and Proposition~\ref{crit}.
	
\subsection{Convergence of Extended SQP and S$\ell^1$QP}
\subsubsection{Sketch of proof of Theorem \ref{th:auslender}}
The proof arguments are adapted from Auslender (2013) \cite[Theorem 3.1, Proposition 3.2]{auslender2013extended}.  Set $l=\inf_Q f$ and recall that $l>-\infty$ by \eqref{eq:ESQM4}. Use first regularity assumptions \eqref{eq:ESQM0}, \eqref{aus-lip} in combination with Lemma~\ref{lem:descent}, to derive that
\begin{align*}
	\frac{1}{\beta_{k+1}}(f(x_{k+1}) - l) + \max_{i = 0, \ldots, m} f_i(x_{k+1}) &\leq \frac{1}{\beta_{k}}(f(x_{k+1}) - l) + \max_{i = 0, \ldots, m} f_i(x_{k+1}) \\
	&\leq \frac{1}{\beta_{k}} (h_{\beta_k}(x_{k+1}, x_k) - l)\\
	&\leq \frac{1}{\beta_{k}} (h_{\beta_k}(x_k, x_k) - l - \frac{\lambda + \beta_k \lambda'}{2} ||x_{k+1} - x_k||^2)\\
	&\leq \frac{1}{\beta_{k}}(f(x_{k}) - l) + \max_{i = 0, \ldots, m} f_i(x_{k}) - \frac{\lambda'}{2} ||x_{k+1} - x_k||^2,
\end{align*}
where the first inequality follows from the monotonicity of $\beta_k$ and the fact that $f(x_{k+1}) - l \geq 0$, the second inequality is due to the descent lemma while the third one is a consequence of the strong convexity of the local model.

The above implies that  $$\frac{1}{\beta_{k+1}}(f(x_{k+1}) - l) + \max_{i = 0, \ldots, m} f_i(x_{k+1})\leq \frac{1}{\beta_{0}}(f(x_{0}) - l) + \max_{i = 0, \ldots, m} f_i(x_{0}),$$ 
thus  $\max_{i = 0, \ldots, m} f_i(x_{k+1})$ is bounded for all $k$ and the compactness assumption \eqref{eq:ESQM3} ensures the boundedness of $x_k$.

Since $\frac{1}{\beta_{k}}(f(x_{k}) - l) + \max_{i = 0, \ldots, m} f_i(x_{k}) \geq 0$, a standard telescopic sum argument gives that $||x_{k+1} - x_k|| \to 0$. Set
$$J_k = \{i=0, \ldots, m: \; \test_i(x_k, x_{k+1}) = \max_{j = 0 \ldots m} \test_j(x_k, x_{k+1}) \},$$
and suppose that $\beta_k \to \infty$. This means that, up to a subsequence, there exists a nonempty set $I \subset \{1, \ldots, m\}$ such that
\begin{align}
	\label{eq:contradictionESQM1}
	&J_{k} = I\\ \nonumber
	&\forall t \in \NN, \forall i \in I, f_i(x_{k}) + \left\langle\nabla f_i(x_k),x_{k+1} - x_k \right\rangle > 0.
\end{align}
Recall that the optimality condition for the local model minimization ensures that, for all $k$, there exists dual variables $u_i \geq 0, i \in J_k$ such that $\sum_{i \in J_k} u_i = 1$ and
\begin{align}
	\label{eq:contradictionESQM2}
	\left\langle \frac{1}{\beta_k}\left(\nabla f(x_k) + (\lambda + \lambda' \beta_k) (x_{k+1} - x_k)\right) + \sum_{i \in J_k} u_i \nabla f_i(x_k), z - x_{k+1} \right\rangle \geq 0,
\end{align}
for any $z \in Q$. Using the boundedness properties of $x_k$ and $u_i$, up to another subsequence, we can pass to the limit in equations (\ref{eq:contradictionESQM1}), (\ref{eq:contradictionESQM2}) to find $\bar{x} \in Q$, $\bar{u}_i, i \in I$ such that
\begin{align*}
	\bar{u}_i &\geq 0\\
	\sum_{i \in I} \bar{u}_i&= 1\\
	f_i(\bar{x}) &\geq 0, i \in I	\\
	\left\langle\sum_{i \in I} \bar{u}_i \nabla f_i(\bar{x}), z - \bar{x} \right\rangle &\geq 0, z \in Q,
\end{align*}
which contradicts qualification assumption \eqref{eq:ESQM2} ($\lim \|x_{k+1} - x_k\|=0$). Therefore, for $k$ sufficiently large, we have
\begin{align*}
	&\beta_k = \beta > 0, \\ 
	&f_i(x_k) + \left\langle \nabla f_i(x_k), x_{k+1} - x_k\right\rangle \leq 0,\\ 
	&0 \in J_k. 
\end{align*}
Given that $x_{k+1} - x_k \to 0$, any accumulation point is feasible. Furthermore, given an accumulation point $\bar{x}$, set $\bar{I} = \{0 \leq i\leq m,\, f_i(\bar{x}) = 0\}$. It must holds that (up to a subsequence) $J_k = \bar{I}$ for a sufficiently large $k$. The fact that $\bar{x}$ is a stationary point follows by passing to the limit in (\ref{eq:contradictionESQM2}).

\subsubsection{Proof of convergence of ESQM}
As granted by Theorem~\ref{th:auslender}, there exists $k_0$ such that $\beta_k=\beta$ for all integer $k\geq k_0$. Since our interest goes to the convergence of the sequence, we may  assume with no loss of generality that $\beta_k$ is equal to $\beta$. Therefore, we only need to consider the behavior of the sequence $\{x_k\}$ with respect to the function 
\begin{align*}
	\Psi_{\beta}(x) = f(x) + \beta \max_{i=0, \ldots, m}(f_i(x)) + i_Q(x),
\end{align*}
whose minimization defines problem $\PP$. Set $\mu=\lambda+\beta\lambda'$, the local model we shall use to study (ESQM) is given by
\begin{align*}
	&h_{\rm ESQM}(x, y) \\
	&= f(x) + \left\langle\nabla f(x), y - x  \right\rangle + \beta \max_{i=0, \ldots, m}(f_i(x) + \left\langle\nabla f_i(x), y - x  \right\rangle) + \frac{\mu}{2} ||y - x||^2, \end{align*}
while the constraints inner approximations reduce to a constant multivalued mapping
$$D(x)=Q.$$


The assumptions   \eqref{cont}  for $D$ are obviously fulfilled.  Let us establish \eqref{majmin}. From assumptions \eqref{eq:ESQM0}, \eqref{aus-lip}, we have for any $x$ and $y$ in $Q$,
\begin{align*}
	f_i(y) &\leq f_i(x) +\left\langle\nabla f_i(x), y - x  \right\rangle + \frac{\lambda'}{2} ||x - y||^2, \quad 0 \leq i \leq m,\\
	f(y) &\leq f(x) +\left\langle\nabla f(x), y - x  \right\rangle + \frac{\lambda}{2} ||x - y||^2.
\end{align*}
Multiplying the first inequalities by $\beta$, taking the maximum over $i$ and adding the the last inequality gives $\Psi_{\beta}(y) \leq h_{\rm ESQM}(x,y)$ for any $x$ and $y$ in $Q$ which yields (i), (iii) \eqref{majmin}. Item (iv) is obvious while  item (iii)  \eqref{majmin} follows from the formula of the subdifferential of a max function Rockafellar and Wets (1998) \cite{rockafellar1998variational}. Assumption $\SSS$ is also fulfilled ($Q$ is convex, hence regular and so is $\Psi_{\beta}$).
 
 Once more the only point that needs to be checked more carefully  is the qualification assumption \eqref{qual}. For all $x, y \in Q$, let 
$I(x,y)$ be the active indices in the definition of $h_{\rm ESQM}(x, y)$. The subdifferential of $\hat h_{ESQM}$ is  given by
\begin{align*}
	& \phantom{==} \partial \hat h_{\rm ESQM}(x, y) \\
	& = \left( 
	\begin{array}{c}
		\mu(x - y) - \nabla^2 f(x) (x - y)\\
		\nabla f(x) + \mu(y - x)
	\end{array}
	\right) + 
	\beta \co \left\lbrace\left(  
	\begin{array}{c}
		- \nabla^2 f_i(x) (x - y)\\
	\nabla f_i(x)
	\end{array}
	\right):  i \in I(x,y) \right\rbrace +
	\left(  
	\begin{array}{c}
		0\\
		N_Q(y)
	\end{array}
	\right),
\end{align*}
where  $\co$ denotes the convex hull. The result follows from the fact that the $f_i$ is $C^2$ and that the hessian are bounded on bounded sets.

	Theorem \ref{th:convergence} applies and gives the desired conclusion. The fact that we eventually obtain a KKT point of $\PP$ is a consequence of Theorem~\ref{th:auslender}.

\subsubsection{Convergence of S$\ell^1$QP}
The proof is quasi-identical to that of ESQP, it is left to the reader.

\section*{Acknowledgments.}
Effort sponsored by the Air Force Office of Scientific Research, Air Force Material Command, USAF, under grant number FA9550-14-1-0056. This research  also benefited from the support of the ``FMJH Program Gaspard Monge in optimization and operations research" and an award of the Simone and Cino del Duca foundation of Institut de France. Most of this work was carried out during the last year of Edouard Pauwels' PhD at Center for Computational Biology in Mines ParisTech (Paris, France) and during a first Postdoctoral stay at LAAS-CNRS (Toulouse, France).

\bibliographystyle{myamsplain}

\bibliography{refs}

\providecommand{\bysame}{\leavevmode\hbox to3em{\hrulefill}\thinspace}
\providecommand{\MR}{\relax\ifhmode\unskip\space\fi MR }
\providecommand{\MRhref}[2]{%
  \href{http://www.ams.org/mathscinet-getitem?mr=#1}{#2}
}
\providecommand{\href}[2]{#2}
\begin{thebibliography}{10}

\bibitem{absil2005convergence}
P.~A. Absil, R.~Mahony, and B.~Andrews, \emph{Convergence of the iterates of
  descent methods for analytic cost functions}, SIAM Journal on Optimization
  \textbf{16} (2005), no.~2, 531--547.

\bibitem{attouch2009convergence}
H.~Attouch and J.~Bolte, \emph{On the convergence of the proximal algorithm for
  nonsmooth functions involving analytic features}, Mathematical Programming
  \textbf{116} (2009), no.~1-2, 5--16.

\bibitem{attouch2010proximal}
H.~Attouch, J.~Bolte, P.~Redont, and A.~Soubeyran, \emph{Proximal alternating
  minimization and projection methods for nonconvex problems: An approach based
  on the Kurdyka-{\L}ojasiewicz inequality}, Mathematics of Operations Research
  \textbf{35} (2010), no.~2, 438--457.

\bibitem{attouch2013convergence}
H.~Attouch, J.~Bolte, and B.~F. Svaiter, \emph{Convergence of descent methods
  for semi-algebraic and tame problems: proximal algorithms, forward--backward
  splitting, and regularized Gauss--Seidel methods}, Mathematical Programming
  \textbf{137} (2013), no.~1-2, 91--129.

\bibitem{auslender2013extended}
A.~Auslender, \emph{An extended sequential quadratically constrained quadratic
  programming algorithm for nonlinear, semidefinite, and second-order cone
  programming}, Journal of Optimization Theory and Applications \textbf{156}
  (2013), no.~2, 183--212.

\bibitem{auslender2010moving}
A.~Auslender, R.~Shefi, and M.~Teboulle, \emph{A moving balls approximation
  method for a class of smooth constrained minimization problems}, SIAM Journal
  on Optimization \textbf{20} (2010), no.~6, 3232--3259.

\bibitem{beck2010gradient}
A.~Beck and M.~Teboulle, \emph{Gradient-based algorithms with applications to
  signal recovery problems}, Convex Optimization in Signal Processing and
  Communications (D.~Palomar and Y.~Eldar, eds.), Cambribge University Press,
  Cambridge, 2010, pp.~42--88.

\bibitem{bert95}
D.~Bertsekas, \emph{Nonlinear programming}, Athena Scientific, Belmont, MA,
  1995.

\bibitem{coste}
J.~Bochnak, M.~Coste, and Roy M.-F., \emph{Real algebraic geometry}, Springer,
  1998.

\bibitem{bolte2007lojasiewicz}
J.~Bolte, A.~Daniilidis, and A.~S. Lewis, \emph{The \L ojasiewicz inequality
  for nonsmooth subanalytic functions with applications to subgradient
  dynamical systems}, SIAM Journal on Optimization \textbf{17} (2007), no.~4,
  1205--1223.

\bibitem{bolte2007clarke}
J.~Bolte, A.~Daniilidis, A.~S. Lewis, and M.~Shiota, \emph{Clarke subgradients
  of stratifiable functions}, SIAM Journal on Optimization \textbf{18} (2007),
  no.~2, 556--572.

\bibitem{bolte2013proximal}
J.~Bolte, S.~Sabach, and M.~Teboulle, \emph{Proximal alternating linearized
  minimization for nonconvex and nonsmooth problems}, Mathematical Programming
  \textbf{146} (2013), no.~1-2, 459--494.

\bibitem{bonnans03}
J.~F. Bonnans, J.~Ch. Gilbert, C.~Lemar\'echal, and C.~Sagastiz\'abal,
  \emph{Numerical optimization: theoretical and practical aspects},
  Springer-Verlag, Berlin, Germany, 2003.

\bibitem{burke}
J.~V. Burke and S.~P. Han, \emph{A robust sequential quadratic programming
  method}, Mathematical Programming \textbf{43} (1989), no.~1-3, 277--303.

\bibitem{nocedal05}
R.~H. Byrd, N.~Gould, J.~Nocedal, and R.~Waltz, \emph{On the convergence of
  successive linear-quadratic programming algorithms}, SIAM Journal on
  Optimization \textbf{16} (2005), no.~2, 471--489.

\bibitem{toint14}
C.~Cartis, N.~Gould, and P.~Toint, \emph{On the complexity of finding
  first-order critical points in constrained nonlinear optimization},
  Mathematical Programming A \textbf{144} (2014), no.~1, 93--106.

\bibitem{chouzenoux}
E.~Chouzenoux, A.~Jezierska, J.~Pesquet, and H.~Talbot, \emph{A
  majorize-minimize subspace approach for $\ell_2-\ell_0$ image
  regularization}, SIAM Journal on Imaging Sciences \textbf{6} (2013), no.~1,
  563--591.

\bibitem{combpesq11}
P.~L. Combettes and J.-C. Pesquet, \emph{Proximal splitting methods in signal
  processing}, Fixed-Point Algorithms for Inverse Problems in Science and
  Engineering (H.H. Bauschke, R.~S. Burachik, P.~L. Combettes, V.~Elser,
  R.~Luke, and H.~Wolkowicz, eds.), Springer Optimization and Its Applications,
  Springer New York, 2011, pp.~185--212.

\bibitem{nemjud}
B.~Cox, A.~Juditsky, and A.~Nemirovski, \emph{Dual subgradient algorithms for
  large-scale nonsmooth learning problems}, Mathematical Programming
  \textbf{148} (2013), no.~1--2, 1--38.

\bibitem{dempster1977maximum}
A.~P. Dempster, N.~M. Laird, and D.~B. Rubin, \emph{Maximum likelihood from
  incomplete data via the EM algorithm}, Journal of the Royal Statistical
  Society, Series B \textbf{39} (1977), no.~1, 1--38.

\bibitem{dontrock}
A.~Dontchev and R.~T. Rockafellar, \emph{Implicit functions and solution
  mappings}, Springer Monograph Series, New York, 2009.

\bibitem{Dries-Miller96}
L.~van~den Dries and C.~Miller, \emph{Geometric categories and o-minimal
  structures}, Duke Mathematical Journal \textbf{84} (1996), no.~2, 497--540.

\bibitem{fletcher85}
R.~Fletcher, \emph{An $\ell^1$ penalty method for nonlinear constraints},
  Numerical optimization (P.~T. Boggs, R.~H. Byrd, and Schnabel~R. B., eds.),
  SIAM, 1985, pp.~26--40.

\bibitem{Fletcher}
\bysame, \emph{Practical methods of optimization, 2nd Edition}, Wiley, 2000.

\bibitem{fletcheretcie}
R.~Fletcher, N.~Gould, S.~Leyffer, P.~Toint, and A.~W\"achter, \emph{Global
  convergence of a trust-region SQP-filter algorithm for general nonlinear
  programming}, SIAM Journal on Optimization \textbf{13} (2002), no.~3,
  635--659.

\bibitem{fukuluotsen03}
M.~Fukushima, Z.~Luo, and P.~Tseng, \emph{A sequential quadratically
  constrained quadratic programming method for differentiable convex
  minimization}, SIAM Journal on Optimization \textbf{13} (2003), no.~4,
  1098--1119.

\bibitem{snopt}
P.~E. Gill, W.~Murray, and M.~Saunders, \emph{SNOPT: An SQP algorithm for
  large-scale constrained optimization}, SIAM Review \textbf{47} (2005), no.~1,
  99--131.

\bibitem{philwong12}
P.~E. Gill and E.~Wong, \emph{Sequential quadratic programming methods}, Mixed
  Integer Nonlinear Programming, The IMA volumes in mathematics and its
  applications (J.~Lee and S.~Leyffer, eds.), vol. 154, Springer New York,
  2012, pp.~147--224.

\bibitem{han77}
S.P. Han, \emph{A globally convergent method for nonlinear programming},
  Journal of Optimization Theory and Applications \textbf{22} (1977), no.~3,
  297--309.

\bibitem{harelewi04}
W.~L. Hare and A.~S. Lewis, \emph{Identifying active constraints via partial
  smoothness and prox-regularity}, Journal of Convex Analysis \textbf{11}
  (2004), no.~2, 251--266.

\bibitem{ioffe2009invitation}
A.~Ioffe, \emph{An invitation to tame optimization}, SIAM Journal on
  Optimization \textbf{19} (2009), no.~4, 1894--1917.

\bibitem{kurdyka1998gradients}
K.~Kurdyka, \emph{On gradients of functions definable in o-minimal structures},
  Annales de l'institut Fourier \textbf{48} (1998), no.~3, 769--783.

\bibitem{lew02}
A.~S. Lewis, \emph{Active sets, nonsmoothness, and sensitivity}, SIAM Journal
  on Optimization \textbf{13} (2002), no.~3, 702--725.

\bibitem{loja1963propriete}
S.~\L{}ojasiewicz, \emph{Une propri\'et\'e topologique des sous-ensembles
  analytiques r\'eels}, Les \'Equations aux D\'eriv\'ees Partielles, vol. 117,
  \'Editions du Centre National de la Recherche Scientifique, 1963, pp.~87--89.

\bibitem{mairal2013optimization}
J.~Mairal, \emph{Optimization with first-order surrogate functions}, ICML
  2013-International Conference on Machine Learning, vol.~28, 2013,
  pp.~783--791.

\bibitem{mara78}
N.~Maratos, \emph{Exact penalty function algorithms for finite dimensional and
  control optimization problems}, Ph.D. thesis, Imperial College, University of
  London, London, U.K, 1978.

\bibitem{nesterov2004introductory}
Y.~Nesterov, \emph{Introductory lectures on convex optimization: A basic
  course}, vol.~87, Springer, 2004.

\bibitem{nocewrig06}
J.~Nocedal and S.~Wright, \emph{Numerical optimization}, Springer Series in
  Operations Research and Financial Engineering, Springer New York, 2006.

\bibitem{noll}
D.~Noll, \emph{Convergence of non-smooth descent methods using the
  Kurdyka-{\L}ojasiewicz inequality}, Journal of Optimization Theory and
  Applications \textbf{160} (2014), no.~2, 553--572.

\bibitem{ortega1970iterative}
J.~M. Ortega and W.~C. Rheinboldt, \emph{Iterative solution of nonlinear
  equations in several variables}, vol.~30, SIAM, 1970.

\bibitem{powe73}
M.~J.~D. Powell, \emph{On search directions for minimization algorithms},
  Mathematical Programming \textbf{4} (1973), 193--201.

\bibitem{rockafellar1998variational}
R.~T. Rockafellar and R.~Wets, \emph{Variational analysis}, vol. 317, Springer,
  1998.

\bibitem{shuzhong1985ekeland}
S.~Shuzhong, \emph{Ekeland's variational principle and the mountain pass
  lemma}, Acta Mathematica Sinica \textbf{1} (1985), no.~4, 348--355.

\bibitem{sol04}
M.~V. Solodov, \emph{On the sequential quadratically constrained quadratic
  programming methods}, Mathematics of Operations Research \textbf{29} (2004),
  no.~1, 64--79.

\bibitem{solo09}
\bysame, \emph{Global convergence of an SQP method without boundedness
  assumptions on any of the iterative sequences}, Mathematical Programming
  \textbf{118} (2009), no.~1, 1--12.

\bibitem{sva02}
K.~Svanberg, \emph{A class of globally convergent optimization methods based on
  conservative convex separable approximations}, SIAM Journal on Optimization
  \textbf{12} (2002), no.~2, 555--573.

\bibitem{wils63}
A.~Wilson, \emph{Simplicial method for convex programming}, Ph.D. thesis,
  Harvard University, 1963.

\bibitem{wrig03}
S.~Wright, \emph{Constraint identification and algorithm stabilization for
  degenerate nonlinear programs}, Mathematical Programming \textbf{95} (2003),
  no.~1, 137--160.

\end{thebibliography}
\end{document}